\documentclass{amsart}
\usepackage{all2021,xypic,color,hyperref}
\newcommand{\FI}{\mathbf{FI}}

\newcommand{\PF}{\mathbf{PF}}
\newcommand{\FIop}{\mathbf{FI^{op}}}
\newcommand{\cC}{\mathcal{C}}
\newcommand{\cF}{\mathcal{F}}
\newcommand{\cE}{\mathcal{E}}

\newcommand{\cN}{\mathcal{N}}
\newcommand{\cM}{\mathcal{M}}
\newcommand{\Mat}{\mathrm{Mat}}
\newcommand{\Sh}{\operatorname{Sh}}
\renewcommand{\phi}{\varphi}
\newcommand{\Gal}{\mathrm{Gal}}
\newcommand{\red}{\mathrm{red}}

\newcommand{\comment}[1]{}

\begin{document}

\title{Components of symmetric wide-matrix varieties}
\author{Jan Draisma, Rob Eggermont, and Azhar Farooq}
\thanks{JD was partially supported by Vici grant
639.033.514 from the Netherlands Organisation for Scientific
Research (NWO) and Project Grant 200021\_191981 from the
Swiss National Science Foundation (SNF).
AF was supported by Vici grant 639.033.514.
RE was supported by Veni grant 016.Veni.192.113 from NWO}

\maketitle

\begin{abstract}
We show that if $X_n$ is a variety of $c \times n$-matrices that
is stable under the group $\Sym([n])$ of column permutations and
if forgetting the last column maps $X_n$ into $X_{n-1}$, then the
number of $\Sym([n])$-orbits on irreducible components of $X_n$ is
a quasipolynomial in $n$ for all sufficiently large $n$. To this end,
we introduce the category of affine $\FIop$-schemes of width one, review
existing literature on such schemes, and establish several new structural
results about them. In particular, we show that under a shift and a localisation,
any width-one $\FIop$-scheme becomes of product form, where $X_n=Y^n$
for some scheme $Y$ in affine $c$-space.  Furthermore, to any $\FIop$-scheme
of width one we associate a {\em component functor} from the category
$\FI$ of finite sets with injections to the category $\PF$ of finite
sets with partially defined maps. We present a combinatorial model for
these functors and use this model to prove that $\Sym([n])$-orbits of
components of $X_n$, for all $n$, correspond bijectively to orbits of a
groupoid acting on the integral points in certain rational polyhedral
cones. Using the orbit-counting lemma for groupoids and theorems on
quasipolynomiality of lattice point counts, this yields our Main
Theorem. We present applications of our methods to counting
fixed-rank matrices with entries in a prescribed set and to counting
linear codes over finite fields up to isomorphism.
\end{abstract}

\section{The main result and background}

\subsection{Main result} \label{ssec:Main}

For a nonnegative integer $n$ we define $[n]:=\{1,\ldots,n\}$.

Let $K$ be a Noetherian ring (commutative with $1$), let $c \in \ZZ_{\geq
0}$, and, for all $n \in \ZZ_{\geq 0}$, let $I_n$ be an ideal in the
polynomial ring $A_n:=K[x_{i,j} \mid i \in [c], j \in [n]]$ such that
the following two conditions are satisfied:
\begin{enumerate}
\item $I_n$ is preserved by the (left) action of the symmetric group
$\Sym([n])$ on $A_n$
via $K$-algebra automorphisms determined by $\pi
x_{i,j}=x_{i,
\pi(j)}$; and
\item $I_n \subseteq I_{n+1}$.
\end{enumerate}

Dually, let $X_n$ be the prime spectrum of $A_n/I_n$, a closed subscheme
of $\Spec(A_n)$. Then the two conditions above express that
\begin{enumerate}
\item $X_n$ is preserved by the induced action of $\Sym([n])$ on $\Spec(A_n)$;
and
\item the projection $\Spec(A_{n+1}) \to \Spec(A_n)$ dual to the
inclusion $A_n \to A_{n+1}$ maps $X_{n+1}$ into $X_n$.
\end{enumerate}
Such a sequence $(X_n)_n$ of schemes of matrices is called
a {\em width-one $\FIop$-scheme} of finite type over $K$ (see
Section~\ref{sec:FISchemes} for a more convenient, functorial definition)
or, more informally, a {\em symmetric wide-matrix scheme}, where the adjective
{\em wide} refers to the fact that $c$ is constant and we are interested
in the case where $n \gg 0$; for brevity, we will usually drop
the adjective {\em symmetric}.

Recall that a {\em quasipolynomial} is a function $f:\ZZ \to \RR$ of
the form
\[ f(n)=c_d(n) n^{d} + c_{d-1}(n) n^{d-1} + \cdots + c_0(n) \]
where each $c_i:\ZZ \to \RR$ is periodic with integral period.
Equivalently, $f$ is a quasipolynomial if and only if there exist an $N$ and
polynomials $f_0,\ldots,f_{N-1}$ such that $f(n)=f_i(n)$
whenever $n \equiv i$ modulo $N$.

\begin{thm}[Main Theorem] \label{thm:Main}
Let $(X_n)_n$ be a width-one $\FIop$-scheme of finite type over a
Noetherian ring $K$. Then the action of $\Sym([n])$ on $X_n$ induces an
action of $\Sym([n])$ on the set $\cC(X_n)$ of irreducible components of
$X_n$, and there exists a quasipolynomial $f:\ZZ \to \RR$ and a natural number $n_0
\in \ZZ_{\geq 0}$ such that the number $|\cC(X_n)/\Sym([n])|$ of
$\Sym([n])$-orbits
on $\cC(X_n)$ equals $f(n)$ for all $n \geq n_0$.
\end{thm}

\subsection{Examples}

We illustrate the Main Theorem by a number of examples. We recall that
the irreducible components of $X_n$ are in one-to-one correspondence
with the inclusion-wise minimal prime ideals $A_n$ that contain $I_n$.

\begin{ex} \label{ex:Planes}
Let $K$ be a domain, take $c=1$, write $x_j$ instead of
$x_{1,j}$,
and let $I_n$ be the ideal generated by all monomials $x_i x_j x_k$
with $i,j,k \in [n]$ distinct. Clearly, the sequence $(I_n)_n$
satisfies the conditions (1) and (2) above.
A prime ideal containing $I_n$ contains at
least one variable from each triple of distinct variables.  Hence the
minimal prime ideals containing $I_n$ are the ideals $I_S:=(\{x_i \mid
i \in S\})$ where $S \subseteq [n]$ is a set of cardinality $n-2$;
the corresponding subscheme is the coordinate plane corresponding to
the coordinates not labelled by  $S$.  Hence $X_n$ has
$\binom{n}{n-2}=\binom{n}{2}$
irreducible components, which form a single orbit under the symmetric
group $\Sym([n])$. The quasipolynomial from the Main Theorem
is $1$. 
\end{ex}

\begin{ex} \label{ex:Points}
Set $K:=\CC$, let $d \in \ZZ_{\geq 0}$, take $c=1$, and let $I_n$ be
the ideal generated by all polynomials $x_i^d-1$ with $i \in [n]$. The
irreducible components of $X_n$ are the points $(\zeta_1,\ldots,\zeta_n)$
where each $\zeta_i$ is an $d$-th root of unity. Thus $X_n$ has $d^n$
irreducible components, and these form $\binom{n+d-1}{d-1}$ orbits under
the group $\Sym([n])$, each of which has a unique representative of the form
\[ (1,\ldots,1,e^{2\pi i/d},\ldots,e^{2\pi i/d},e^{2\cdot
2\pi i/d},\ldots,e^{2 \cdot 2\pi i/d},\ldots,e^{(d-1)\cdot 2
\pi i/d},\ldots,e^{(d-1) \cdot 2 \pi i/d}), \]
where the numbers of occurrences of $e^{j\cdot 2\pi i/d},\
j=0,\ldots,d-1$ are
arbitrary nonnegative integers whose sum is $n$. 
\end{ex}

\begin{ex} \label{ex:Galois}
Set $K:=\CC(t)$, where $t$ is a variable, let $c=1$, and let $I_n$
be the ideal generated by all polynomials of the form $x_i^2-t$ with
$i \in [n]$. Then $I_1$ is a prime ideal, but for $n \geq 2$ and two
distinct $i,j \in [n]$ any prime ideal $P$ containing $I_n$ also contains
$(x_i^2-t)-(x_j^2-t)=(x_i-x_j)(x_i+x_j)$ and hence either $x_i-x_j$
or $x_i+x_j$. Hence, modulo $P$, the variables $x_1,\ldots,x_n$ can be
partitioned into two subsets: within each of these sets, all variables
are equal modulo $P$, and they are minus the variables in the other set,
again modulo $P$.

Conversely, if $A,B \subseteq [n]$ are disjoint sets with $A \cup B=[n]$,
then the ideal $P_{A,B}$ generated by the polynomials $x_i-x_j$ with
$(i,j) \in (A \times A) \cup (B \times B)$, the polynomials $x_i+x_j$
with $(i,j) \in A \times B$, and the polynomial $x_1^2-t$ is a minimal
prime ideal over $I_n$. By the above, these are all minimal primes over
$I_n$. Note that $P_{A,B}=P_{C,D}$ if and only if $\{A,B\}=\{C,D\}$,
so there is a bijection between unordered partitions of $[n]$ into two
parts (one of which may be empty).

The number of unordered partitions of $[n]$ is $2^{n-1}$. The action
of $\Sym([n])$ on minimal primes over $I_n$ corresponds to the natural
action of $\Sym([n])$ on unordered partitions of $[n]$.
The
number of $\Sym([n])$-orbits on the latter is $\lfloor n/2\rfloor +1$---indeed,
$\{A,B\},\{\tilde{A},\tilde{B}\}$ are in the same $\Sym([n])$-orbit if and only
if $\min\{|A|,|B|\}=\min\{|\tilde{A}|,|\tilde{B}|\}$, and this number
takes any of the values in $\{0,\ldots,\lfloor n/2 \rfloor\}$. Here the
quasipolynomial is the function $f(n)=\lfloor n/2 \rfloor + 1$, and it holds
for all $n \geq 0$. 
\end{ex}

\begin{ex} \label{ex:cube}
Set $K:=\CC$, $c:=1$, let $d \in \ZZ_{\geq 1}$, and let $I_n$ be the
ideal generated by all differences $x_i^d-x_j^d$ with $i,j \in [n]$. In
this case, each prime ideal $P$ containing $I_n$ also contains, for each
$i \neq j$, a polynomial of the form $x_i - \zeta x_j$, where $\zeta$
is a $d$-th root of unity. Hence the variables can be partitioned into
$d$ sets, where modulo $P$ the variables in each set are $e^{2\pi i/d}$
times the variables in the previous set.

Conversely, let $A=(A_0,\ldots,A_{d-1})$ be an ordered partition of
$[n]$: a sequence of disjoint, potentially empty
subsets of $[n]$ whose union is $[n]$. Then the ideal $P_A$ generated
by the polynomials $x_j-e^{2 b \pi i/d} x_l$ whenever $l$ lies in some $A_a$ and $j$
lies in $A_{a+b}$ (indices modulo $d$) is a minimal prime over $I_n$,
and by the above all primes arise in this manner. Furthermore, $P_A=P_B$
if and only if the sequence $B$ arises from $A$ by a cyclic permutation,
i.e., by adding an element of $\ZZ/d\ZZ$ to the indices.

Hence the irreducible components of the scheme defined by $I_n$ correspond
bijectively to orbits of ordered partitions of $[n]$ into $d$ parts
under the action of $\ZZ/d\ZZ$ by rotation of the parts. The number
of $\Sym([n])$-orbits on such components is therefore equal to the number of
$(\ZZ/d\ZZ)\times \Sym([n])$-orbits on ordered partitions.
Modding out $\Sym([n])$ first,
what remains is to count the $\ZZ/d\ZZ$-orbits on ordered {\em
integer} partitions of $n$ into $d$ nonnegative parts. This is done using
the orbit-counting lemma (due to Cauchy, Frobenius, and not
Burnside) for $\ZZ/d\ZZ$: for
$e \in \ZZ/d\ZZ$ define $f(e):=\gcd(d,e) \in \{1,\ldots,d\}$. Then rotation by
$e$ on such integer partitions has the same fixed points as
rotation by $f(e)$. This number is $0$ if
$n$ is not divisible
by $d/f(e)$, and equal to $\binom{n/(d/f(e)) + f(e) -
1}{f(e)-1}$ otherwise: the first $f(e)$ positions in the
partition can be
filled arbitrarily with nonnegative integers whose sum is
$n/(d/(f(e)))$, and this determines the partition fixed under
rotating over $f(e)$. Thus the number of $\Sym([n])$-orbits on
components equals
\[ \frac{1}{d} \sum_{e \in \ZZ/d\ZZ: (d/f(e))|n}
\binom{n/(d/f(e)) + f(e) - 1}{f(e)-1}, \]
which is, indeed, a quasipolynomial in $n$.  
\end{ex}

These examples illustrate different aspects of the proof of
Theorem~\ref{thm:Main}. First, the projection $X_{n+1} \to X_n$ maps each
irreducible component of $X_{n+1}$ into some component of $X_n$, but not
necessarily {\em onto} some such component: in Example~\ref{ex:Planes},
the coordinate planes involving the variable $x_{n+1}$ are mapped onto
coordinate {\em lines} rather than planes. However, ``most'' coordinate
planes are mapped onto coordinate planes. We will capture these relations
between components of the $X_n$ as $n$ varies by the so-called {\em
component functor} (see Section~\ref{sec:ComponentFunctor}), which is a
contravariant functor from $\FI$ to the category $\PF$ of finite sets
with partially defined maps. This functor plays a fundamental role in
the proof of Theorem~\ref{thm:Main}, and also yields a more detailed
picture of the components of the $X_n$ for varying $n$.

Second, Example~\ref{ex:Points} illustrates that, while the number of
components of $X_n$ can grow exponentially with $n$, the number of orbits
is upper-bounded by a polynomial.

Third, in Example~\ref{ex:Galois} we see that, if we adjoin $\sqrt{t}$
to the ground field $K$, then the example reduces to a variation on
Example~\ref{ex:Points}: there are $2^n$ components and $n+1$ orbits
on components. This suggests that the quasipolynomiality in the Main
Theorem is due to the action of a Galois group. We will see that this is,
indeed, the case when the wide-matrix scheme is of {\em product
type}; see \S\ref{ssec:FiniteOverWide}.

Finally, Example~\ref{ex:cube} shows that even when $K$ is an
algebraically closed field, quasipolynomiality (rather than
polynomiality) occurs. In part,
this is because we will have to work over larger base fields that are
transcendental extensions of the ground field $K$; and in part, it is
because Galois groups are not the sole reason for quasipolynomiality:
in \S\ref{ssec:SecondQuasi}, we will replace the orbit-counting via
Galois groups to orbit-counting via certain groupoids.

We conclude this subsection with two interesting applications of our
techniques.

\begin{cor}
Let $K$ be a field, let $S \subseteq K$ be a finite subset, and let $k$
be a natural number. For every $n \in \ZZ_{\geq 0}$, define
\[ M_n:=\{A \in S^{n \times n} \subseteq K^{n \times n} \mid
\rk(A) = k\}, \]
the set of all rank-$k$ matrices all of whose entries are in $S$.
Let $\Sym([n])$ act by simultaneous row and column permutations on
$M_n$. Then $|M_n/\Sym([n])|$ is a quasipolynomial in $n$ for $n \gg
0$.
\end{cor}

\begin{proof}
Consider the morphism $\phi_n: \AA_K^{k \times n} \times \AA_K^{k
\times n} \to \AA_K^{n \times n}$ given by $(A,B) \mapsto A^T
\cdot B$. For each $C \in M_n$, the closed subscheme
$\phi_n^{-1}(C)$ is irreducible---indeed, for any field extension $L$ of $K$,
its $L$-points
form an orbit under the action of the group $\GL_k(L)$
acting via $g(A,B)=((g^{-1})^TA,gB)$, and irreducibility
follows from irreducibility of the group scheme $\GL_k$.
Hence $(X_n)_n$ defined by
$X_n:=\phi_n^{-1}(M_n)$ is a wide-matrix scheme with $c=2k$
whose irreducible components are in $\Sym(n)$-equivariant one-to-one
correspondence with the points of $M_n$.
The corollary follows from the Main Theorem to $(X_n)_n$.
\end{proof}

\begin{re}
The same argument works for {\em symmetric} rank-$k$ matrices and for
{\em skew-symmetric} rank-$k$ matrices. More generally, by a similar
argument: if $Y_n \subseteq \AA_K^{n \times n}$ is a closed subscheme of
the variety of rank-$\leq k$ matrices such that $Y_n$ is preserved under
the action of $\Sym([n])$ by conjugation and such that forgetting the last
row and column maps $Y_{n+1}$ to $Y_n$, then, too, the number of orbits
of $\Sym([n])$ on irreducible components of $Y_n$ is a quasipolynomial
in $n$ for $n \gg 0$. 
\end{re}

\begin{ex}
Consider $K=\QQ$ and let $M_{k,n}$ be the set of {\em symmetric} $n \times
n$-matrices with entries in $\{0,1\}$ of rank precisely $k$.
Then:
\begin{itemize}
\item $M_{0,n}$ consists of the zero matrix only, so there
is a single $S_n$-orbit.
\item Each $S_n$-orbit in $M_{1,n}$ has a unique representative of the form
\[ \begin{bmatrix} J & 0 \\ 0 & 0 \end{bmatrix} \]
where $J$ is an $m \times m$-matrix ($1 \leq m \leq n$) with all ones, and the
zeros are block matrices of appropriate sizes. Hence there
are $n$ orbits.
\item There are three types of $S_n$-orbits in $M_{2,n}$,
with representatives
\[ \begin{bmatrix} J_1 & 0 & 0 \\ 0 & J_2 & 0 \\ 0 & 0 & 0
\end{bmatrix},
\quad
\begin{bmatrix} 0 & J & 0 \\ J^T & 0 & 0 \\ 0 & 0 & 0 \end{bmatrix},
\quad
\begin{bmatrix} J_1 & J_2 & 0 \\ J_2^T & 0 & 0 \\ 0 & 0 & 0
\end{bmatrix}\]
where, in the first case, $J_1,J_2$ are all-one matrices of formats
$n_1 \times n_1,n_2 \times n_2$ with $1 \leq n_1 \leq n_2$ and $n_1+n_2
\leq n$; in the second case, $J$ is an all-one $n_1 \times n_2$-matrix with $1
\leq n_1 \leq n_2$ and $n_1+n_2 \leq n$; and in the third case, $J_1$
is an all-one $n_1 \times n_1$-matrix and $J_2$ is an
all-one $n_1 \times n_2$-matrix
with $1 \leq n_1,n_2$ and $n_1 + n_2 \leq n$.

In the last case, the number of pairs $(n_1,n_2)$ is $\binom{n}{2}$.
In the first and second cases, we have to count pairs $(n_1,n_2)$ with
$1 \leq n_1 \leq n_2$ and $n_1+n_2 \leq n$. This number equals
\[ \sum_{n_1=1}^{\lfloor n/2 \rfloor} (n-2n_1+1)=\lfloor n/2 \rfloor
\cdot \lceil n/2 \rceil. \]
Summarising, the number of $S_n$-orbits on $M_{2,n}$ equals
\[ 2 \cdot \lfloor n/2 \rfloor \cdot \lceil n/2 \rceil
+ \binom{n}{2}, \]
clearly a quasipolynomial in $n$. 
\end{itemize}
\end{ex}

For the next application, recall that a {\em linear code} of length
$n$ and dimension $m$ over $\FF_q$ is a linear subspace of $\FF_q^n$
of dimension $m$. {\em Puncturing} such a code means deleting a
coordinate; we only allow this when the dimension
of the code does not drop. There is a natural notion of isomorphism of
linear codes, involving permuting and scaling coordinates as well as
applying an automorphism of $\FF_q$; see Example~\ref{ex:Codes}.

\begin{thm} \label{thm:Codes}
Fix a finite field $\FF_q$ and a natural number $m$. Let $\cC$ be a family
of isomorphism classes of $m$-dimensional codes, of varying lengths,
that is preserved under puncturing. Then for $n \gg 0$ the number of
length-$n$ elements in $\cC$ is a quasipolynomial in $n$.
\end{thm}

This theorem is not a direct consequence of our Main Theorem, but
rather of one of the tools that we develop for the Main Theorem, see
\S\ref{ssec:ModelFunctors} and \S\ref{ssec:SecondQuasi}.

\subsection{Relations to existing literature}

The functorial viewpoint and the notions of $\FI$-algebras and
$\FIop$-schemes that we will use are strongly influenced by the literature
on $\FI$-modules \cite{Church12,Church14}. Furthermore, the insight
that counting combinatorial objects is best done through {\em
species}---functors from the category of finite sets with bijections
to itself---is due to \cite{Joyal81}. While we use neither nontrivial
results about $\FI$-modules nor nontrivial results about species, this
paper could not have been written without this background.

Given a wide-matrix scheme $(X_n)_n$ over $K$, one can define the
inverse limit $X_\infty:=\lim_{\ot n} X_n$. This limit is a variety in
the space of $c \times \NN$-matrices and preserved by the action of the
symmetric group $\Sym(\NN)$ permuting columns. Furthermore, for $n \gg 0$, $X_n$
is in fact the image of $X_\infty$ under projection---this follows
from Proposition~\ref{prop:Nicefy}---and so, as we are
interested in properties of $X_n$ for large $n$, $X_\infty$ contains all
relevant information.  Much literature on wide-matrix schemes uses this
set-up---a $\Sym(\NN)$-invariant subvariety of an infinite-dimensional
affine space---rather than the functorial set-up. But, as we will see,
for counting purposes the functorial set-up is more convenient.

The descending chain property of wide-matrix schemes (see
Section~\ref{ssec:Noetherianity}) was first established
in \cite{Cohen67,Cohen87} and then rediscovered in
\cite{Aschenbrenner07,Hillar09}, and used to prove the Independent
Set Theorem in algebraic statistics in the latter paper. A further
application to algebraic statistics is \cite{Draisma08b}. Images of
wide-matrix schemes under monomial maps also satisfy the descending
chain property \cite{Draisma13a}.

These Noetherianity results admit proofs using Gr\"obner methods in the
spirit of \cite{Sam14}, which can be turned into explicit algorithms. A
special-purpose algorithm was used in \cite{Brouwer09e} to find the
defining equations for the Gaussian two-factor model, a general-purpose
algorithm was implemented in Macaulay 2 \cite{Hillar13b}.  The results
of the current paper are also effective: there exists an algorithm that,
on input the finitely many equations defining a wide-matrix scheme,
computes the quasi-polynomial from the Main Theorem. But we believe
that it is unlikely that a general-purpose algorithm for this will ever
be implemented.

The sequence of ideals $(I_n)_n$ defining a wide-matrix scheme has a Hilbert
function in two variables: one for the degree and one for $n$. It
turns out to be a rational function of a very specific form
\cite{Gunturkun18,Krone17,Nagel17}; and the same holds for
finitely generated modules over an $\FI$-algebra that is
finitely generated in width one \cite{Nagel20}.

Further commutative algebra for wide-matrix schemes was developed in
\cite{Nagel17b}, where Noetherianity of finitely generated modules
over their coordinate rings was established; \cite{VanLe18}, where
the codimension of $X_n$ in its ambient space and the projective
dimension of $I_n$ in its ambient polynomial ring are studied; and
\cite{VanLe18b}, which concerns the Castelnuovo regularity of $I_n$. A
beautiful, as yet open conjecture from the latter two papers, is that
the projective dimension and the regularity both are precisely a linear
function of $n$ for $n \gg 0$. For codimension, this is established
in \cite[Theorem 3.8]{VanLe18}; it also follows from our work (see
Theorem~\ref{thm:DimComp}). The paper \cite{Juhnke20}
contains a number of open problems; Problem 7.3, which asks
about the behaviour of the primary decomposition, is close
in spirit to our more geometric result.

Noetherianity implies, roughly speaking, that each wide-matrix scheme is
a finite union of irreducible wide-matrix schemes. Here irreducibility
does not mean that each individual $X_n$ is irreducible, but rather
that $X_\infty$ is irreducible in the topology in which the closed
subsets are $\Sym(\NN)$-stable closed subsets of the space of $c
\times \NN$-matrices. For instance, the wide-matrix scheme where $c=1$
and $X_n=\{0,1\}^n$ turns out to be irreducible in this setting. In
\cite{Nagpal20}, these irreducible varieties are classified for $c=1$.
For general $c$, the algebraic and semi-algebraic geometry of such
symmetric subvarieties are studied in \cite{Kummer22};
Example~\ref{ex:Kummer} below comes from that paper. 

\subsection{Organisation of this paper}

This paper is organized as follows. $\FIop$-schemes are introduced
in Section 2. In particular, we define wide-matrix spaces in this
context. Every width-one $\FIop$-scheme of finite type is isomorphic to a
closed $\FIop$-subscheme of a wide-matrix space (see Lemma \ref{lm:Prod}).

In Section 3, we prove that any width-one $\FIop$-scheme of
finite type is of {\em product type} after a suitable shift and
a suitable localisation (see the Shift Theorem \ref{thm:Shift} and
Proposition~\ref{prop:Shift}). The shifting technique is also used by
the first author in his work on polynomial functors \cite{Draisma17},
and the Shift Theorem is reminiscent of, and was inspired by, the Shift
Theorem in \cite{Bik21}. It is the strongest new
structural result that we prove about wide-matrix schemes.

In Section 4, for an $\FIop$-scheme $X$, we define the component functor
$\cC_X$ from $\FI$ to the category $\PF$ of finite sets with partially
defined maps. The component functor is one of the most important notions
of this paper, and in the remainder of the paper we obtain an almost
complete combinatorial description of $\cC_X$ in the case where $X$
is a width-one $\FIop$-scheme.

Section 5 is devoted to the proof of the Main Theorem. In three steps,
we construct more and more refined combinatorial models for $\cC_X$.
The first ones, called {\em elementary model functors}, allow
us to prove the Main Theorem when $X$ is of product type
(see \S\ref{ssec:FiniteOverWide}). By the Shift Theorem this
situation is always attained by a shift and a localisation, and to undo the simplifications
caused by that shift and localisation, we need the two more complicated combinatorial
models dubbed {\em model functors} (see \S\ref{ssec:ModelFunctors})
and {\em pre-component functors} (see \S\ref{ssec:PreComponent}). A
major generalisation in the step from elementary model functors to model
functors is that we pass from counting orbits under a finite
group---in the application to $\cC_X$, this is the image of a Galois
group---{\em which is part of the defining data} of an
elementary model functor, to counting
orbits under a finite groupoid {\em which emerges by itself} from the defining
data of a model functor. The proof of Theorem~\ref{thm:ModelFunctor}
that model functors have a quasipolynomial count is entirely elementary,
but very subtle. A direct application is
Theorem~\ref{thm:Codes}.

In comparison, the step from model functors
to pre-component functors is conceptually small. In \S\ref{ssec:FinalQuasi} we prove
that pre-component functors always have a quasipolynomial count, and in
\S\ref{ssec:PreCompFIopScheme} we establish that the component functor
of a wide-matrix scheme satisfies the properties Compatibility
(1)--(3) of a pre-component
functor. This, then, completes the proof of the Main Theorem.

\section{Width-$1$ $\FIop$-schemes} \label{sec:FISchemes}

In this section we collect fundamental facts about $\FI$-algebras and
$\FIop$-schemes.

\subsection{The category $\FI$}

The category $\FI$ has as objects finite sets, and for $S,T \in \FI$
the hom-set $\Hom_\FI(S,T)$ is the set of {\em injections} $S \to T$. The
category $\FIop$ is its opposite category.

\subsection{$\FI$-algebras and $\FIop$-schemes}

Let $K$ be a ring (commutative, with $1$). All $K$-algebras $A$ are
required to be commutative, have a $1$, and the homomorphism $K \to A$
is required to send $1$ to $1$. Homomorphisms of $K$-algebras are
unital ring homomorphisms $A \to B$ compatible with the homomorphisms from $K$
into them.

\begin{de}
An $\FI$-algebra $B$ over $K$ is a covariant functor from $\FI$ to the
category of $K$-algebras with unital $K$-algebra homomorphisms. Dually,
$B$ gives rise to a {\em contravariant} functor $X$ from $\FI$ to the
category of affine schemes over $K$. To remind ourselves of the
contravariance of this functor, we call such a functor an affine
$\FIop$-scheme over $K$.

A morphism $B \to A$ of $\FI$-algebras over $K$ is a natural
transformation from $B$ to $A$: it consists of a $K$-algebra
homomorphism $\phi(S):B(S) \to A(S)$ for all $S$ such that for
all $S,T \in \FI$ and $\pi \in
\Hom_{\FI}(S,T)$ the following diagram commutes.
\[
\xymatrix{
B(S) \ar[r]^{\phi(S)} \ar[d]_{B(\pi)} & A(S) \ar[d]^{A(\pi)} \\
B(T) \ar[r]_{\phi(T)} & A(T).
}
\]
Morphisms of affine $\FIop$-schemes are defined dually. Since
we will only consider affine schemes, we will sometimes drop the
adjective ``affine''.
\end{de}

\begin{re} \label{re:B0}
If $B$ is an $\FI$-algebra over $K$, then we write $B_n:=B([n])$.  In fact,
$B$ is then also an $\FI$-algebra over $B_0$: for each finite set $S$ the
unique inclusion $\emptyset \to S$ yields an algebra homomorphism $B_0
\to B(S)$, and functoriality implies that the $K$-algebra homomorphisms
$B(S) \to B(T)$ corresponding to injections $S \to T$ are compatible
with the $B_0$-algebra structure. 
\end{re}

The classical equivalences of categories $B \to \Spec(B)$ and $X
\mapsto K[X]$ between $K$-algebras and affine schemes over $K$
yield equivalences of categories between $\FI$-algebras and affine
$\FIop$-schemes. Given an $\FI$-algebra $B$ over $K$, we write $\Spec(B)$
for the $\FIop$-scheme $S \mapsto \Spec(B(S))$ and given an affine
$\FIop$-scheme $X$, we write $K[X]$ for the $\FI$-algebra $S \mapsto
K[X(S)]$.

An {\em ideal} in an $\FI$-algebra $B$ has the obvious definition: it
consists of an ideal $I(S)$ for each $S \in \FI$ such that for each $\pi
\in \Hom_\FI(S,T)$ the map $B(\pi):B(S) \to B(T)$ maps $I(S)$ into $I(T)$.

Given an $\FI$-algebra $B$, for each $S \in \FI$ the symmetric group
$\Sym(S)$ acts from the left on $B(S)$: indeed, $\Sym(S)=\Hom_{\FI}(S,S)$,
and each $\pi$ in the latter set yields a $K$-algebra homomorphism
$B(\pi): B(S) \to B(S)$. The axioms expressing that $B$ is a functor
imply that $(\pi,b) \mapsto B(\pi)(b)$ is a left action of $\Sym(S)$
by $K$-algebra automorphisms on $B(S)$.

Similarly, given an $\FIop$-scheme $X$, for each $S \in \FI$ the
symmetric group $\Sym(S)$ acts on $X(S)$ by automorphisms of affine
$K$-schemes. When acting on points of $X(S)$ with values in a $K$-algebra
$L$, i.e., on the set of $K$-algebra homomorphisms $K[X(S)] \to L$,
this is a naturally a {\em right} action, reflecting the
fact that $X$ is a contravariant functor.

Tensor products of $\FI$-algebras over $K$ are defined in the
straightforward manner, and they correspond to products in the category
of affine $\FIop$-schemes over $K$.

\begin{re} \label{re:Sn}
The symmetric group $\Sym([n])$ acts on $B_n$ by $K$-algebra
automorphisms, and the map $B(\iota):B_n \to B_{n+1}$, where $\iota:[n]
\to [n+1]$ is the standard inclusion, is a $\Sym([n])$-equivariant $K$-algebra
homomorphism, if $\Sym([n])$ is regarded as the subgroup of $\Sym([n+1])$ consisting
of all permutations that fix $n+1$. Conversely, from the data (for all
$n$) of $B_n$,
the action of $\Sym([n])$ on $B_n$, and the $\Sym([n])$-equivariant map $B_n \to
B_{n+1}$ the $\FI$-algebra $B$ can be recovered up to isomorphism. This
gives another, more concrete picture of $\FI$-algebras similar to that
used in \S\ref{ssec:Main}. However,
the definition of $\FI$-algebras as a functor from $\FI$
to $K$-algebras is more elegant and, as we will see, often
more convenient. 
\end{re}

\subsection{Base change} \label{ssec:BaseChange}

\begin{de}
If $B$ is an $\FI$-algebra over a ring $K$, and $L$ is a $K$-algebra,
then we obtain an $\FI$-algebra $B_L$ over $L$ by setting $S \mapsto
L \otimes_K B(S)$. In the special case where $L$ is the localisation
$K[1/h]$ for some $h \in K$, we also write $B[1/h]$ for $B_L$.

Dually, if $X=\Spec(B)$ is the associated $\FIop$-scheme, then we write
$X_L=\Spec(B_L)$ for the base change, and $X[1/h]$ if $L=K[1/h]$.
\end{de}

\subsection{Wide-matrix spaces}

In this paper, the following $\FI$-algebras and
$\FIop$-schemes play a prominent role.

\begin{ex} \label{ex:AK}
Let $A_K$ be the $\FI$-algebra that maps $S$ to $K[x_{j} \mid j
\in S]$ and $\pi \in \Hom_\FI(S,T)$ to the $K$-algebra homomorphism
determined by $x_j \mapsto x_{\pi(j)}$. For each $c \in \ZZ_{\geq 0}$,
$A_K^{\otimes c}$ (where the tensor product is over $K$) is isomorphic to,
and will be identified with, the $\FI$-algebra over $K$ that maps $S$ to
$K[x_{i,j} \mid i \in [c], j \in S]$. Write $\Mat_{c,K}:=\Spec(A_K^{\otimes
c})$. If $L$ is a $K$-algebra, then the {\em set of
$L$-points}
$\Mat_{c,K}(L)$
is the contravariant functor from $\FI$ to sets that assigns to $S$
the set $L^{c \times S}$ of $c \times S$-matrices over $L$, and to a morphism $\pi:S \to T$
the map $L^{c \times T} = (L^c)^T \to (L^c)^S = L^{c \times S}$ where
the middle map is composition with $\pi$. The $\FIop$-scheme
$\Mat_{c,K}$ is called the {\em wide-matrix space over $K$
with $c$ rows} or, less precisely, a {\em wide-matrix space}.
\end{ex}

\subsection{Width}

Let $B$ be an $\FI$-algebra over $K$.

\begin{de}
For $S \in \FI$ and $b \in B(S)$, we call the minimal $n$ such that $b$
lies in $\pi B([n])$ for some $\pi \in \Hom_{\FI}([n],S)$ the {\em width}
of $b$, denoted $w(b)$.  
\end{de}

\begin{ex}
Assuming that $K$ is not the zero ring, the element $x_{1,1} + x_{2,4}
+ x_{3,4}^2 \in A_K^{\otimes 3}([4])$ has width $2$: it is the image
of $x_{1,1} + x_{2,2} + x_{3,2}^2$ under $A_K^{\otimes 3}(\pi)$ where
$\pi:[2] \to [4]$ is defined by $1 \mapsto 1$ and $2 \mapsto
4$.  
\end{ex}

Note that the width satisfies $w(b_1 + b_2),w(b_1
\cdot b_2) \leq w(b_1)+w(b_2)$.

\subsection{$\FI$-algebras finitely generated in width $\leq 1$}

\begin{de}
Let $B$ be an $\FI$-algebra. Let $S_i, i \in I$ be a collection of objects
in $\FI$ and for each $i \in I$ let $b_i$ be an element of $B(S_i)$.
There is a unique smallest $\FI$-subalgebra $R$ of $B$ such that $R(S_i)
\ni b_i$ for all $i$. This is called the {\em $\FI$-algebra generated
by the $b_i$}. Concretely, $R(S)$ is the $K$-subalgebra of $B(S)$
generated by all elements $\pi(b_i)$ for all $i \in I$ and $\pi \in
\Hom_\FI(S_i,S)$. 
\end{de}

\begin{de}
An $\FI$-algebra $B$ over $K$ is {\em finitely generated} if there exists
a finite collection $(b_i \in B(S_i))_{i \in I}$ of elements that
generates $B$. 
\end{de}

\begin{de}
An $\FI$-algebra $B$ over $K$ is {\em generated in width $\leq w$}
if there exists a collection $(b_i \in B(S_i))_{i \in I}$ of
elements of width $\leq w$ that generates $B$. 
\end{de}

\begin{re}
Our notion of {\em width} is called {\em degree} in the $\FI$-module
literature, e.g. in Definition \cite[Definition 2.14]{Church12}. We
reserve the word degree for degrees of monomials and polynomials.
\end{re}

We will be mostly interested in $\FI$ algebras that are
finitely generated in width $\leq 1$ in the following sense.

\begin{de}
An $\FI$-algebra $B$ over $K$ is {\em finitely generated in width $\leq
w$} if $B$ is finitely generated and generated in width $\leq w$. It is
straightforward to see that this is equivalent to the condition that
$B$ is generated by a finite collection of elements $b_i \in B(S_i)$
of width $\leq w$. 
\end{de}

Put differently yet, given an $\FI$-algebra, recall that $B_n$ is
defined as $B([n])$. Then $B$ is finitely generated in width $\leq w$
if and only if $B_0,B_1,\ldots,B_w$ are finitely generated $K$-algebras
and $B$ is generated by these as an $\FI$-algebra over $K$.

\begin{ex}
The algebra $A_K$ and its tensor power $A^{\otimes c}$ are finitely
generated in width $\leq 1$, namely, by the elements $x_{i,1} \in
A^{\otimes c}([1])$ with $i=1,\ldots,c$. 
\end{ex}

We now introduce the main characters of our paper.

\begin{de}
An {\em $\FIop$-scheme of width one}, or {\em width-one
$\FIop$-scheme}, of finite type over $K$
is the spectrum of an $\FI$-algebra over $K$ finitely generated in
width $\leq 1$. 
\end{de}

The class of $\FI$-algebras finitely generated in width at most $1$
is closed under taking finite direct sums and tensor products over
$K$. Dually, the corresponding class of schemes is closed under disjoint
unions and Cartesian products.

The following lemma will be useful later.

\begin{lm} \label{lm:Prod}
Let $B$ an $\FI$-algebra over $K$ finitely generated in width $\leq 1$,
let $X=\Spec(B)$ be the corresponding width-one $\FIop$-scheme of finite type
over $K$, and set $Z:=X([1])$. Then for each $S \in \FI$ the map $X(S)
\to \prod_{j \in S} X(\{j\}) \cong Z^{S}$, where the product is over
$\Spec(B_0)$, is a closed embedding. Furthermore, the
$\FIop$-scheme $Z^S$ is
isomorphic to a closed $\FIop$-subscheme of $\Mat_{c,B_0}$ for
some $c$.
\end{lm}

\begin{proof}
Dually, we need to show that the map $\bigotimes_{j \in S} B(\{j\})
\to B(S)$, where the tensor product is over $B_0$ and where $B\{j\}
\to B(S)$ comes from the inclusion $\{j\} \to S$, is surjective. This
follows from the fact that $B$ is generated in width at most $1$. For the
last statement, note that $B_1$ is finitely generated as a $K$-algebra,
hence {\em a fortiori} as a $B_0$-algebra. If $B_1$ is generated by $c$
elements over $B_0$, then $B$ is a quotient of $A_{B_0}^{\otimes c}$,
$Z$ is a closed subscheme of the $c$-dimensional affine space
$\AA^c_{B_0}$ over $B_0$, and $S \mapsto Z^S$ a closed
$\FIop$-subscheme of $\Mat_{c,B_0}$.
\end{proof}

\subsection{Noetherianity} \label{ssec:Noetherianity}

The following result is by now classical, and the starting point of
a growing body of literature on $\FI$-algebras.

\begin{thm} \label{thm:Noetherianity}
Let $K$ be a Noetherian ring. Then every $\FI$-algebra $B$ over $K$ that
is finitely generated in width $\leq 1$ is Noetherian, i.e., if $I_1
\subseteq I_2 \subseteq \cdots$ is an ascending chain of ideals in $B$,
then $I_n=I_{n+1}$ for all $n \gg 0$.
\end{thm}

A Gr\"obner basis proof of a closely related theorem---formulated for an
infinite symmetric group acting on an infinite-dimensional polynomial
ring---first appeared in \cite{Cohen67,Cohen87} and was rediscovered
in \cite{Aschenbrenner07} (for $c=1$) and \cite{Hillar09} (for general
$c$). In the current set-up, the result follows from the reference
in the following remark.

\begin{re}
In fact, $B$ is Noetherian in a stronger sense: any finitely generated
$B$-module satisfies the ascending chain condition on
submodules \cite[Theorem 6.15]{Nagel17b}.
\end{re}

\subsection{Nice width-one $\FIop$-schemes}

The following consequence of Noetherianity will be useful to us: it implies
that when we are interested in the tail of a width-one $\FIop$-scheme $X$
of finite type over a Noetherian ring $K$, i.e., in $X([n])$ for $n \gg
0$, then we may without loss of generality assume that the map $X([n+1])
\to X([n])$ dual to the inclusion $[n] \to [n+1]$ is dominant for all
$n$.

\begin{prop} \label{prop:Nicefy}
Let $K$ be a Noetherian ring and let $X=\Spec(B)$ be an affine width-one
scheme of finite type over $K$. Then there exists an $n_0 \in \ZZ_{\geq
0}$ such that for all $S,T \in \FI$ with $n_0 \leq |S| \leq |T|$ and
all $\pi \in \Hom_\FI(S,T)$, the homomorphism $B(\pi):B(S) \to B(T)$
is injective. Define
\[ B'(S):=\begin{cases}
		B(S) & \text{if $|S| \geq n_0$, and}\\
		B(S)/\ker(B(\sigma)) & \text{if
		$|S| \leq n_0$}
\end{cases} \]
where $\sigma$ is any chosen element of $\Hom_\FI(S,[n_0])$ (the
result doesn't depend on $\sigma$). For any $S,T \in \FI$ and $\pi
\in \Hom_\FI(S,T)$ the $K$-algebra homomorphism $B(\pi): B(S) \to B(T)$
induces a well-defined $K$-algebra homomorphism $B'(\pi):B'(S) \to B'(T)$,
and thus $B'$ becomes an $\FI$-algebra over $K$, finitely generated in
width $\leq 1$, with the property that $B'(\pi)$ is injective
for all $\pi \in \Hom_\FI(S,T)$. Set $X':=\Spec(B')$; then $X'(\pi):X'(T)
\to X'(S)$ is dominant for all $\pi \in \Hom_\FI(S,T)$.
\end{prop}

\begin{proof}
First we show that for all $S,T \in \FI$ we have $\ker B(\pi) =
\ker B(\sigma)$ for all $\pi , \sigma \in \Hom_{FI} (S,T)$. For
any two injections $\pi, \sigma : S \rightarrow T$ there exists a
permutation $\tau$ of $T$ such that $\pi = \tau \circ \sigma $. For
$f \in \ker B(\sigma)$ we have $B(\pi)(f)=B(\tau \circ \sigma)(f)=
B(\tau)(B(\sigma)(f))= B(\tau)(0)=0$. This implies that $\ker B(\sigma)
\subset \ker B(\pi)$. By symmetry, also the reverse inclusion holds. In
particular this shows that $\ker B(\sigma)$ is independent of the choice
of $\sigma$ and it is stable under the action of the group $\Sym(S)$.

Now suppose that the first claim of the proposition is not true, that
is, there does not exist such an $n_0$. Then there exists a strictly
increasing sequence of positive integers $(m_i)_i$ and injections $\pi_i :
[m_i] \rightarrow [m_i+1]$ such that for all $i$, $\ker B(\pi_i)$ is not
trivial. Let $I_i$ be the $\FI$-ideal in $B$ generated by $\bigcup_{j=1}^i
\ker B(\pi_j)$; by the first paragraph, $I_i(T)=\{0\}$ for all $T$ with
$|T|>m_i$. Hence the sequence $(I_i)_i$ is a strictly increasing chain
of $\FI$-ideals of $B$; this is a contradiction to the fact that $B$
is Noetherian.

Let $S,T \in \FI$ and let $\pi \in \Hom_{\FI}(S,T)$. If $|S| \geq n_0$,
then $B'(S)=B(S)$ and it is immediate that $B(\pi):B(S) \to B(T)$ induces
a $K$-algebra homomorphism $B'(S) \to B'(T)$. Otherwise, let $\sigma:S
\to [n_0]$ be an injection, so that $B'(S)=B(S)/\ker B(\sigma)$. If $|T|
\geq n_0$, then $\pi$ factors via $\sigma$ and it follows that $\ker
B(\pi) \supseteq \ker B(\sigma)$; again, $B(\pi)$ induces a map $B'(S)
\to B'(T)=B(T)$. Finally, if also $|T| \leq n_0$, then let $\iota:T
\to [n_0]$ be an injection. Replace $\sigma$ by $\iota \circ \pi$,
another injection $S \to [n_0]$. Then $\ker B(\sigma)=\ker (B(\iota)
\circ B(\pi))$ by the first paragraph, and hence $B(\pi)$ maps $\ker
B(\sigma)$ into $\ker B(\iota)$, so that, once more, it induces a map
$B'(S) \to B'(T)$.

The check that $B'$ is an $\FI$-algebra over $K$ finitely generated in
width $\leq 1$ is straightforward, and the check that each $B'(\pi)$
is injective follows from a similar analysis to that in the previous
paragraph. The final statement is standard: injective $K$-algebra
homomorphisms yield dominant morphisms.
\end{proof}

\begin{de}
We call an $\FI$-algebra $B$ over a ring $K$ {\em nice}
if for all $\pi \in \Hom_{\FI}(S,T)$ the map $B(\pi):B(S)
\to B(T)$ is injective; also its spectrum is then called
nice. Proposition~\ref{prop:Nicefy} says that if $K$ is Noetherian,
then any width-one affine $\FIop$-scheme of finite type over $K$ agrees
with a nice scheme for sufficiently large $S$.
\end{de}

\begin{lm} \label{lm:LocNice}
Let $B$ be a nice $\FI$-algebra over $K$ and let $h \in K$.
Then $B[1/h]$ is a nice $\FI$-algebra over $K[1/h]$.
\end{lm}

\begin{proof}
For each $\pi \in \Hom_\FI(S,T)$, $B[1/h](\pi)$ is the $K[1/h]$-algebra
homomorphism $B(S)[1/h] \to B(T)[1/h]$ obtained by the $K$-algebra
homomorphism $B(S) \to B(T)$ by localisation. By assumption, the latter is
injective. Hence, since localisation is an exact functor from $K$-modules
to $K[1/h]$-modules, so is the former.
\end{proof}

\subsection{Reduced $\FIop$-schemes}

\begin{de}
The $\FI$-algebra $B$ over $K$ is called {\em reduced} if $B(S)$ has no
nonzero nilpotent elements for any $S \in \FI$. Then also $X=\Spec(B)$
is called reduced.
\end{de}

The following lemma is immediate.

\begin{lm} \label{lm:Reduced}
Let $B$ be an $\FI$-algebra over $K$ and for each $S \in \FI$ let
$B^\red(S)$ be the quotient of $B(S)$ by the ideal of nilpotent
elements. Then for $\pi \in \Hom_\FI(S,T)$ the homomorphism $B(\pi)$
induces a homomorphism $B^\red(S) \to B^\red(T)$, and this makes $B^\red$
into a reduced $\FI$-algebra over $K$. Furthermore, if $B$ is finitely
generated in width $\leq 1$, then so is $B^\red$.
\hfill $\square$
\end{lm}

It follows that, to prove our Main Theorem, we may always assume that $X$
is reduced.

\subsection{Shifting}

The idea of shifting an $\FI$-structure over a finite set goes back to
\cite{Church12}. The first author's also used it in his work on topological 
Noetherianity of polynomial functors \cite{Draisma17}, except that
there, one shifts over a vector space. 

\begin{de}
Let $S_0$ be a finite set. Then $\Sh_{S_0}:\FI \to \FI$ is the functor
that sends $S$ to the disjoint union $S_0 \sqcup S$ and $\pi \in \Hom_\FI(S,T)$ to $\Sh_{S_0}
\pi:S_0 \sqcup S \to S_0 \sqcup T$ that is the identity on $S_0$ and equal
to $\pi$ on $S$. For an $\FI$-algebra $B$ over $K$ we write $\Sh_{S_0}
B:=B \circ \Sh_{S_0}$ and for the affine $\FIop$-scheme $X=\Spec(B)$
over $K$ we write $\Sh_{S_0} X:=X \circ \Sh_{S_0} = \Spec(\Sh_{S_0}
B)$. Furthermore, for a homomorphism $\phi:B \to R$ of $\FI$-algebras
over $K$, we write $\Sh_{S_0} \phi$ for the morphism $\Sh_{S_0} B \to
\Sh_{S_0} R$ that sends $S$ to $\phi(S_0 \sqcup S)$, and similarly
for morphisms of affine $\FIop$-schemes. A straightforward check shows
that $\Sh_{S_0}$ is a covariant functor from $\FI$-algebras over $K$
into itself and from affine $\FIop$-schemes over $K$
into itself.
\end{de}

If $B$ is finitely generated in width $\leq 1$, then so is $\Sh_{S_0}
B$; and hence, if $X$ is a width-one $\FIop$-scheme of finite type over
$K$, then so is $\Sh_{S_0} X$.

\begin{re}  \label{re:Basis}
If $B':=\Sh_{S_0} B$, then $B'$ is naturally an $\FI$-algebra over
$B'_0=B(S_0)$ (see Remark~\ref{re:B0}). Thus
shifting naturally leads to a change of base ring---informally, by
shifting we ``move some functions into the constants''. For an $f \in
B(S_0)$, its image in $B(S_0 \sqcup S)$ under $B(\iota)$, where $\iota$
is the natural injection $S_0 \to S_0 \cup S$, will also be denoted
simply by $f$. This is slight abuse of notation, especially as $B(\iota)$
needs not be an injection if $B$ is not nice, but this will not lead
to confusion.

In the interpretation from Remark~\ref{re:Sn} of $\FI$-algebras consisting
of algebras acted upon by $\Sym([n])$ with suitable maps between them,
one may model shifting by restricting the action to the subgroup of
$\Sym([n])$ that fixes the numbers $1$ up to $n_0:=|S_0|$. We will,
however, not explicitly use this model.
\end{re}

For future use, we note that $\Sh_{S_0}(\Sh_{S_1} B)$ is canonically
isomorphic to $\Sh_{S_0 \sqcup S_1} B$, and similarly for $\FIop$-schemes.
Furthermore, shifting preserves reducedness and niceness.

\section{The Shift Theorem}
\label{sec:Finiteness}

\subsection{Formulation of the Shift Theorem}

Recall from Lemma~\ref{lm:Prod} that a width-one $\FIop$-scheme
$X=\Spec(B)$ of finite type over a ring $K$ is a closed $\FIop$-subscheme
of $S \mapsto Z^S$, where $Z=X([1])$ and where the product is over
$B_0$.  In this section we establish the fundamental result that
in fact, after a suitable shift and localisation, $X$ becomes {\em equal}
to such a product.

\begin{thm}[Shift Theorem] \label{thm:Shift}
Let $B$ be a reduced and nice $\FI$-algebra that is finitely generated in
width $\leq 1$ over a ring $K$, assume that $1 \neq 0$ in $B_0$, and set
$X:=\Spec(B)$. Then there exists an $S_0 \in \FI$ and a nonzero element $h
\in B(S_0)$ such that $X':=(\Sh_{S_0} X)[1/h]$ is isomorphic to $S \mapsto
Z^S$, where $Z=X'([1])$ and where the product is over $B(S_0)[1/h]$.
\end{thm}

\subsection{Shift-and-localise}

Before proving the Shift Theorem, we establish that
shifting and localisation commute in a suitable sense.

\begin{lm} \label{lm:ShiftAndLocalise}
Let $B$ be a reduced $\FI$-algebra over $K$, $S_0,S_1 \in \FI$, $h_0 \in
B(S_0)$ nonzero, $B':=(\Sh_{S_0} B)[1/h_0]$, $h_1 \in B'(S_1)$ nonzero,
and $B'':=(\Sh_{S_1} B')[1/h_1]$. Then there exists a nonzero $h \in
B(S_0 \sqcup S_1)$ such that $(\Sh_{S_0 \sqcup S_1} B)[1/h]
\cong B''$ as $\FI$-algebras over $K$.
\end{lm}

\begin{proof}
By multiplying $h_1$ with a suitable power of the image of $h_0$ in
$B'(S_1)$, we achieve that $h_1$ lies in the image of $B(S_0 \sqcup S_1)$
in $B(S_0 \sqcup S_1)[1/h_0]=B'(S_1)$. Let $\tilde{h}_1$ be an element
of $B(S_0 \sqcup S_1)$ mapping to $h_1$. Then, by a straightforward
computation, $h:=h_0 \tilde{h}_1 \in B(S_0 \sqcup S_1)$ does the trick.
\end{proof}

\subsection{Proof of the Shift Theorem}

\begin{proof}
By Lemma~\ref{lm:Prod}, $X$ is (isomorphic to) a closed $\FIop$-subscheme
of $\Mat_{c,B_0}$ for some $c$. Let $R:S \to B_0[x_{ij} \mid i \in
[c], j \in S]$ be the coordinate ring of the latter wide-matrix space,
and let $I$ be the ideal of $X$ in $R$.

Fix any monomial order on $\ZZ_{\geq
0}^c$. We will use this order to compare monomials in the variables
$x_{1j},\ldots,x_{cj}$ for any $j$.

Elements of $I([1])$ are $B_0$-linear combinations of monomials
$x_{1,1}^{\alpha_1} \cdots x_{c,1}^{\alpha_c}$ with $\alpha \in \ZZ_{\geq
0}^c$. Let $M \subseteq \ZZ_{\geq 0}^c$ be the set of (exponent vectors
of) leading monomials of {\em monic} elements of $I([1])$. By Dickson's
lemma, there exist finitely many monic elements $f_1,\ldots,f_k \in
I([1])$ whose leading monomials $u_1,\ldots,u_k$ generate $M$.

Now there are two possibilities. Either for every $n \in \ZZ_{\geq 1}$
and every nonzero $f \in I([n]) \subseteq R([n])$, some monomial in
$f$ is divisible by $R(\pi)u_i$ for some $i \in [k]$ and some $\pi \in
\Hom_\FI([1],[n])$---or not. In the former case, using that the $f_i$
are monic, we can do division with remainder by the $R(\pi)f_i$ until the
remainder is zero, and it follows that $f_1,\ldots,f_k \in I([1])$
generate the $\FI$-ideal $I$. Then $X$ itself is a product
as desired---indeed, by Lemma~\ref{lm:Prod}, $X$ is a closed
$\FIop$-subscheme of the $\FIop$-scheme $S \mapsto X([1])^S$, and the
fact that the $\FI$-ideal of $X$ is generated by $I([1])$ implies that
the corresponding closed embedding is an isomorphism. Hence in this case we can
take $S_0:=\emptyset$ and $h:=1 \neq 0 \in B_0$.

In the latter case, let $n_0$ be minimal such that there exists a
nonzero $f \in I([n_0])$ none of whose terms are divisible by any
$R(\pi)u_i$. Regard $f$ as a polynomial in $x_{1,n_0},\ldots,x_{c,n_0}$
with coefficients in $R([n_0-1])$, let $u=x_{1,n_0}^{\alpha_1} \cdots
x_{c,n_0}^{\alpha_c}$ be the leading monomial of $f$, and let $\tilde{h}
\in R([n_0-1])$ be the coefficient of $u$ in $f$. Now $\tilde{h} \not
\in I([n_0-1])$ by minimality of $n_0$ and the fact that no term in
$\tilde{h}$ is divisible by any $R(\pi)u_i$ with $i \in [k]$ and $\pi \in
\Hom_{\FI}([1],[n-1])$---indeed, such a term, multiplied with $u$, would
yield a term in $f$ with the same property. Set $S_0:=[n_0-1]$ and let $h$
be the image of $\tilde{h}$ in $B(S_0)$; this is nonzero by construction.

Now set $B':=(\Sh_{S_0} B)[1/h]$ and $X':=\Spec(B')$, and note that $1
\neq 0$ in $B'_0$. Then $X'$ is a
closed $\FIop$-subscheme of $\Mat_{c,B'_0}$, and we claim that if we
construct $M' \subseteq \ZZ_{\geq 0}^c$ for $X'$ in the same manner as
we constructed $M$ for $X$, then $M' \supsetneq M$. Indeed, if $\iota:[1]
\to S_0 \sqcup [1]$ is the natural inclusion, then $R(\iota)$ maps $f_i$
to an element in the ideal of $X(S_0 \sqcup [1])$ with the same leading
monomial $u_i$, and this maps to an element of the ideal of $X'([1])$
with that same leading monomial. This shows that $M' \supseteq M$.
Furthermore, via the bijection $\tau:[n_0] \to S_0 \sqcup [1]$ that is
the identity on $S_0=[n_0-1]$ and sends $n_0$ to $1$ we obtain another
element $R(\tau)f$ in the ideal of $X(S_0 \sqcup [1])$, whose image in the
ideal of $X'([1])$ has an invertible leading coefficient (namely, $h$)
and leading monomial $x_{1,1}^{\alpha_1} \cdots x_{c,1}^{\alpha_c}$. We
thus find that $\alpha \in M'$, while $\alpha \not \in M$ by construction.

The fact that $B$ is nice and reduced implies that so is $B'$. Hence we can
continue in the same manner with $B'$.  By Dickson's lemma, the set $M$
can strictly increase only finitely many times. Hence after finitely
many shift-and-localise steps, we reach the former case, where we know
that $X$ is a product.

Finally, we invoke Lemma~\ref{lm:ShiftAndLocalise} to conclude that this
finite sequence of shift-and-localise steps can be turned into a single
shift followed by a single localisation inverting a nonzero element.
\end{proof}

We will use the following strengthening of the Shift Theorem
in the case where $K$ is Noetherian.

\begin{prop} \label{prop:Shift}
In the setting of the Shift Theorem, if we further assume that $K$
is Noetherian, then there exists a
nonzero $h' \in B'_0$ such that $B'':=B'[1/h']$ and $X'':=\Spec(B'')$
have the following properties:
\begin{enumerate}
\item like in the Shift Theorem, $X''$ is isomorphic to $S \mapsto V^S$ where
$V:=X''([1])$ and where the product is over $B''_0$;
\item $B''_0$ is a domain; and
\item for each $S \in \FI$, every irreducible component of $V^S$ maps
dominantly into $\Spec(B''_0)$.
\end{enumerate}
\end{prop}

\begin{proof}
The $\FIop$-scheme $X'=\Spec(B')$ from the Shift Theorem maps $S$
to $Z^S$, where $Z=X'([1])$ and where the product is over $B'_0$.
By construction, $B'$ is reduced, nice, and $1 \neq 0$ in $B'_0$. Any
localisation by a nonzero $h' \in B'_0$ satisfies (1). We will now
construct $h'$ so as to satisfy (2) and (3).

As $K$ is Noetherian and $B'_0$ is a finitely generated $K$-algebra,
$B'_0$ is Noetherian. Hence $\Spec(B'_0)$ is the union of finitely many
irreducible components; let $C$ be one of them. Then there exists a
nonzero $h_1 \in B'_0$ that vanishes identically on all other irreducible
components of $\Spec(B'_0)$. Now $B'_0[1/h_1]$ is a domain, namely,
the coordinate ring of $C[1/h_1]$.

Furthermore, $B_1'[1/h_1]$ is a finitely generated $B_0'[1/h_1]$-algebra
and by generic freeness \cite[Theorem 14.4]{Eisenbud95}, there exists a
nonzero $h_2 \in B_0'[1/h_1]$ such that $B_1'[1/h_1][1/h_2]$ is a free
$B_0'[1/h_1][1/h_2]$-module. After multiplying with a power of (the image
of) $h_1$, we may assume that $h_2$ the image of some $\tilde{h}_2 \in
B_0$. Then set $h':=h_1 \tilde{h}_2$.

Set $B'':=B'[1/h']$ and $X'':=\Spec(B'')=X'[1/h']$. Now $B''_0$ is a
localisation of the domain $B'[1/h_1]$, hence a domain, so
(2) holds.

Furthermore, for every $S \in \FI$, $X''(S)$ is the product over
$B''_0=B'[1/h']$ of $|S|$ copies of $V:=X''([1])$. Its coordinate
ring $B''(S)$ is then a tensor product over $B''_0$ of $|S|$ copies of
the free $B''_0$-module $B''_1$, and hence $B''(S)$ is itself a free
$B''_0$-module. Furthermore, again since niceness is preserved, the map
$B''_0 \to B''(S)$ is injective. Then, by the going-down theorem for
flat extensions \cite[Lemma 10.11]{Eisenbud95}, every minimal prime ideal of $B''(S)$ intersects $B''_0$
in the zero ideal, so that every irreducible component of $X''(S)$
maps onto $\Spec(B''_0)$, as desired.
\end{proof}

\begin{de} \label{de:ProductType}
Let $L$ be a Noetherian domain, $Q \supseteq L$ a ring extension such that $Q$
is a finitely generated $L$-algebra and free as an $L$-module. Set
$Z:=\Spec(Q)$. Then the $\FIop$-scheme over $L$ defined by $S \mapsto
Z^S$, where the product is over $L$, is said to be {\em of product
type}. As we have seen above, each irreducible component of $Z^S$ then
maps dominantly into $\Spec L$.
\end{de}

In Section~\ref{sec:Proof} we will establish our Main Theorem for
$\FIop$-schemes of product type and then relate the general case to the
product case via the Shift Theorem.

\begin{ex} \label{ex:Kummer}
To illustrate the Shift Theorem and Proposition~\ref{prop:Shift} we
Analyse \cite[Example 3.20]{Kummer22} in the case of curves. In our
notation, let $X_d(S)$ be the reduced,
closed subvariety of $\Mat_{2,\CC}(S)$ sonsisting of all $S$-tuples
of points $(x_i,y_i)$ for which there exists a nonzero degree-$\leq d$
polynomial $p \in \CC[x,y]$ with $p(x_i,y_i)=0$ for all $i \in S$. It
is proved in \cite{Kummer22} that $X_d(S)$ is an irreducible variety for all $d \geq 1$ and
all $S$, so it is not particularly interesting from the perspective of
counting components. However, it {\em is} interesting from the perspective
of the Shift Theorem. Take $n_0:=\dim \CC[x,y]_{\leq d} - 1$, so that
through $n_0$ points $(x_i,y_i),\ i=1,\ldots,n_0$ in general
position
goes a unique plane curve $C$ of degree $\leq d$. The coefficients
of the corresponding polynomial $p$ are rational functions of the
$(x_i,y_i)$ with $i \in [n_0]$. Take for $h$ a common multiple of the
denominators of these rational functions, so that $C$ is a
curve over the ring
$\CC[x_1,y_1,\ldots,x_{n_0},y_{n_0}][1/h]=:B_0'$.
Then $X'=(\Sh_{[n_0]}X)[1/h]$
is the $\FIop$-variety that maps $S$ to $C^S$, where the
product is over $B_0'$.  
\end{ex}

\section{The component functor}
\label{sec:ComponentFunctor}

To establish the Main Theorem, we will analyse the functor that assigns
to a finite set $S$ the set of components of $X(S)$. This
functor takes values in another category called $\PF$.

\subsection{Contravariant functors $\FI \to \PF$}

\begin{de}
Let $\PF$ be the category whose objects are finite sets and whose
morphisms $T \to S$ are partially defined maps from $T$ to $S$, i.e.,
maps $\pi$ into $S$ whose domain $\dom(\pi)$ is a subset of $T$. If $\pi:T
\to S$ and $\sigma:S \to U$ are morphisms in this category, then $\sigma
\circ \pi$ is defined precisely at those $i \in T$ for which $i \in
\dom(\pi)$ and $\pi(i) \in \dom(\sigma)$; and $\sigma
\circ \pi$ takes the value $\sigma(\pi(i))$ there.
\end{de}

We will be interested in contravariant functors $\cF:\FI \to \PF$ and
morphisms between these.

\begin{de} \label{de:Morphism}
A {\em morphism} from a contravariant functor $\cF:\FI \to \PF$ to another
such functor $\cF'$ is a collection of everywhere defined maps
$(\Psi(S):\cF(S) \to \cF'(S))_{S \in \FI}$ such that for all $S,T \in
\FI$ and $\pi \in \Hom_\FI(S,T)$ the diagram
\[
\xymatrix{
\cF(T) \ar[r]^{\Psi(T)} \ar[d]_{\cF(\pi)} & \cF'(T) \ar[d]^{\cF'(\pi)}\\
\cF(S) \ar[r]_{\Psi(S)} & \cF'(S)
}
\]
commutes in the following sense: if the leftmost map
$\cF(\pi)$ is defined at some $f \in \cF(T)$, then the rightmost
map $\cF'(\pi)$ is defined at $\Psi(T)(f)$, and we have
$\cF'(\pi)(\Psi(T)(f))=\Psi(S)(\cF(\pi)(f))$. The morphism is called
injective/surjective if each $\Psi(S)$ is injective/surjective, and an
isomorphism if each $\Psi(S)$ is bijective and moreover $\cF'(\pi)$ is
defined {\em precisely} at all $f' \in \cF'(T)$ such that $\cF(\pi)$
is defined at $\Psi(T)^{-1}(f')$. 
\end{de}

Note that our morphisms $\cF \to \cF'$ are not precisely natural
transformations, since we do not require that the diagram above commutes
as a diagram of partially defined maps: we allow the partially defined
map $\cF'(\pi) \circ \psi(T)$ to have a larger domain than $\psi(S)
\circ \cF(\pi)$.

\subsection{The component functor of an $\FIop$-scheme}

\begin{de} \label{de:ComponentFunctor}
Let $B$ be a finitely generated $\FI$-algebra over a Noetherian ring
$K$, so that $X=\Spec(B)$ is an affine $\FIop$-scheme of finite type
over $K$.  Then we define the contravariant functor $\cC_X: \FI \to \PF$
on objects by
\[ \cC_X(S)=\{\text{the irreducible components of $X(S)$}\} \]
and on morphisms $\pi \in \Hom_{\FI}(S,T)$ as follows: $\cC_X(\pi)$
is defined at some component $C \in \cC_X(T)$ if (and only if) $X(\pi):X(T) \to X(S)$
maps $C$ dominantly into a component of $X(S)$. The functor $\cC_X$ is
called the {\em component functor} of $X$. 
\end{de}

Note that the condition that $K$ is Noetherian and $B$ is finitely
generated implies that, indeed, $\cC_X(S)$ is a finite set for each $S$.

\begin{ex}
In Example~\ref{ex:Planes}, $\cC_X$ is isomorphic to the functor $\FI \to
\PF$ that assigns to the set $S$ the set $\binom{S}{2}$ of two-element
subsets and to $\pi:S \to T$ the partially defined map $\binom{T}{2}
\to \binom{S}{2}$ that sends $\{i,j\}$ to $\{\pi^{-1}(i),\pi^{-1}(j)\}$
whenever this is defined. 
\end{ex}

In the definition of the component functor we have not assumed that $B$
is gene\-rated in width $\leq 1$, and indeed larger $\FI$-algebras
also yield
interesting examples.

\begin{ex} \label{ex:Cycles}
Let $K$ be a field and let $R$ be the $\FI$-algebra that assigns to $S$
the ring $R(S)=K[x_{i,j} \mid i,j \in S]/(\{x_{i,j}-x_{j,i} \mid i,j
\in S\})$ and to a morphism $\pi \in
\Hom_\FI(S,T)$ the $K$-algebra homomorphism determined by $x_{i,j}
\mapsto x_{\pi(i),\pi(j)}$. This $\FI$-algebra is generated in width
$2$ by the two elements $x_{1,1},x_{1,2} \in R([2])$.

It is well known that this $\FI$-algebra
is {\em not} Noetherian; the following example is closely related
to \cite[Example 3.8]{Hillar09}. Let $I_d(S)$ be the ideal generated
by all {\em cycle monomials} of the form $x_{i_1,i_2} x_{i_2,i_3}
\cdots x_{i_k,i_1}$ where $3 \leq k \leq d$ and $i_1,\ldots,i_k$ are
distinct. Then $I_3 \subsetneq I_4 \subsetneq \ldots$ is an infinite
strictly increasing chain of ideals in $R$. Let $I_\infty$ be
their union, and let $X=\Spec(R/I_\infty)$. A prime ideal $P$ in $R(S)$
containing $I_\infty(S)$ contains at least one variable from every
cycle of length at least $3$, so the edges $\{i,j\}$ corresponding to
variables $x_{i,j}$ with $i\neq j$ that are {\em not} in $P$ form a
forest with vertex set $S$. Every forest is contained in a tree with
vertex set $S$. Correspondingly, every such tree $T$ gives rise to a
minimal prime ideal containing $I_\infty(S)$, namely the ideal generated by all
$x_{i,j}$ with $\{i,j\}$ not an edge in $T$.

It follows that the minimal prime ideals
of $R(S)/I_\infty(S)$ are in bijection to the trees
with vertex set $S$. Recall that, by Cayley's
formula, this number of trees is $n^{n-2}$ when $n:=|S| \geq
2$. In particular, the number of $\Sym([n])$-orbits is at least $(n^{n-2})/n!$
and hence superpolynomial in $n$; this shows that in the Main Theorem
the width-one condition cannot be dropped.

Furthermore, given a $\pi \in \Hom_\FI(S,T)$, $\cC_X(\pi)$ is defined
on trees $\Delta$ with vertex set $T$ as follows. If the induced
subgraph of $\Delta$ on $\pi(S)$ is connected (and hence a tree), then
$\cC_X(\pi)(\Delta)$ is that tree but with the label $j \in \pi(S)$
replaced by $\pi^{-1}(j)$. Otherwise, $\cC_X(\pi)$ is not defined at
$\Delta$. 
\end{ex}

\subsection{The underlying species}

In our proof of the Main Theorem, we will give a fairly complete
picture of the component functor of width-one $\FIop$-schemes, at least
on sets $S \in \FI$ with $|S| \gg 0$. The first observation is that
for {\em any} contravariant functor $\cF:\FI \to \PF$ and any $\pi \in
\End_{\FI}(S)=\Sym(S)$, $\cF(\pi)$ is defined everywhere on $\cF(S)$,
and a bijection there. After all, by the properties of a contravariant
functor $\id_{\cF(S)} = \cF(\pi \circ \pi^{-1})= \cF(\pi^{-1}) \circ
\cF(\pi)$. It follows that the functor from the category of finite
sets with bijections to itself that sends $S$ to $\cF(S)$ and $\pi$
to $\cF(\pi)^{-1}$ is a covariant functor and hence a {\em species} in
the sense of \cite{Joyal81}; we call this the {\em underlying species} of the
$\cF$. For the Main Theorem it would suffice to know
the underlying species of the component functor $\cC_X$ of $X$. However,
to understand this species, we will also need to have some information on
the partially defined maps $\cC_X(\pi)$ where $\pi:S \to T$ is {\em not}
a bijection.

\subsection{A property in width one}

The second observation on component functors concerns width-one $\FIop$-schemes.

\begin{lm}
Suppose that $X$ is a width-one affine $\FIop$-scheme of finite type over
a Noetherian ring $K$. Then there exists an $n_0$ such that for all $S,T
\in \FI$ with $n_0 \leq |S|\leq |T|$ and all $\pi \in \Hom_\FI(S,T)$,
the partially defined map $\cC_X(\pi)$ is {\em surjective}.
\end{lm}

\begin{proof}
Take the $n_0$ from Proposition~\ref{prop:Nicefy}, so that $X(\pi):X(T)
\to X(S)$ is dominant for all $S,T \in \FI$ with $n_0 \leq |S| \leq
T$. Then for each component of $X(S)$ there must be some component of
$X(T)$ mapping dominantly into it.
\end{proof}

\section{Proof of the Main Theorem} \label{sec:Proof}

In this section, which takes up the remainder of the paper, we establish the Main Theorem. To
do so, on the one hand we develop purely combinatorial tools (see
\S\S\ref{ssec:Elementary},\ref{ssec:FirstQuasi},\ref{ssec:OrbitCounting},\ref{ssec:ModelFunctors},\ref{ssec:SecondQuasi},\ref{ssec:PreComponent},\ref{ssec:FinalQuasi},)
and on the other hand we establish algebraic
results relating the component functors of
width-one $\FIop$-schemes to those combinatorial tools (see
\S\S\ref{ssec:CompWideSpace},\ref{ssec:FiniteOverWide},\ref{ssec:PreCompFIopScheme}).
Finally, all is combined in \S\ref{ssec:Proof} to establish the Main
Theorem. We would like to highlight \S\ref{ssec:SecondQuasi}, where
from a so-called model functor we extract certain groupoids
acting on unions of
rational cones, after which we use an orbit-counting lemma for groupoids
from \S\ref{ssec:OrbitCounting} to establish quasipolynomiality in that
crucial case.

\subsection{The component functor of a wide-matrix space}
\label{ssec:CompWideSpace}

Let $L$ be a finitely generated $K$-algebra, where $K$ is a Noetherian
ring, and $c\in \ZZ_{\geq 0}$. Then for each $S \in \FI$ we have a
natural morphism $\Mat_{c,L}(S) \to \Spec(L)$ (corresponding to the natural embedding $L \to L[\Mat_{c,L}(S)]$), and the preimages
of the irreducible components of $L$ are the irreducible components
of $\Mat_{c,L}(S)$. This establishes the following.

\begin{lm} \label{lm:WideMatrixSpace}
Let $k$ be the number of minimal prime ideals of $L$.  The component
functor of $\Mat_{c,L}$ is isomorphic to the functor that assigns
the set $[k]$ to each $S \in \FI$ and the identity on $[k]$ to each
$\pi \in \Hom_{\FI}(S,T)$. In particular, the number of $\Sym([n])$-orbits on
$\cC_{\Mat_{c,L}}([n])$ is equal to $k$. \hfill $\square$
\end{lm}

\subsection{Elementary model functors} \label{ssec:Elementary}

We construct a class of contravariant functors $\FI \to \PF$ from which,
as we will see, the component functor of a width-one $\FIop$-scheme of
finite type over a Noetherian ring is built up in a suitable sense.

\begin{de} \label{de:Elementary}
Let $k \in \ZZ_{\geq 0}$, let $G$ be a subgroup of $\Sym([k])$, and, for $S
\in \FI$, let $\cE(S)$ be a subset of $[k]^S$ that is preserved under the
diagonal action of $G$ on $[k]^S$. Assume, furthermore, that for all $\pi
\in \Hom_\FI(S,T)$ the map $[k]^T \to [k]^S,\ \alpha \mapsto \alpha
\circ \pi$ maps $\cE(T)$ into
$\cE(S)$. Then the contravariant functor $\FI \to \PF$ that sends $S$ to
$\cE(S)/G$ and $\pi \in \Hom_{\FI}(S,T)$ to the (everywhere defined) map
\[ \cE(T)/G \to \cE(S)/G,\quad G\cdot \alpha \mapsto G
\cdot (\alpha \circ \pi) \]
is called an {\em elementary model functor} $\FI \to \PF$. Note that the
latter map is well-defined as $G$ acts diagonally.
\end{de}

\subsection{A first quasipolynomial count}
\label{ssec:FirstQuasi}

\begin{prop} \label{prop:BurnsideStanley}
Let $k \in \ZZ_{\geq 0}$, $G$ a subgroup of $\Sym([k])$ and $M$ a
$G$-stable downward closed subset of $\ZZ_{\geq 0}^k$; that is, for all $\beta \in M$ and $g \in G$ we have $g\beta \in M$, and for
all $\beta \in M$ and $j \in [k]$ with $\beta_j >0$ we have $\beta - e_j \in M$, where $e_j$ is the $j$-th
standard basis vector in $\ZZ^k$. Then there exists a quasipolynomial
$f$ such that for $n \gg 0$ the number of $G$-orbits on the set $M_n$
of elements $\beta \in M$ of {\em total degree} $|\beta|:=\sum_j \beta_j$
equal to $n$ equals $f(n)$.
\end{prop}

\begin{proof}
By the orbit-counting lemma, that number of orbits equals
\[ \frac{1}{|G|} \sum_{g \in G} |M_n^g| \]
where $M_n^g=\{\alpha \in M_n \mid g\alpha=\alpha\}$. So it suffices to
prove that each of the summands is a quasipolynomial for $n \gg 0$.

The set $M$ has a so-called {\em Stanley decomposition}
\cite{Stanley82}
\[ M=\bigsqcup_{i=1}^d (\alpha_i + \ZZ_{\geq 0}^{I_i}) \]
for suitable subsets $I_i \subseteq [k]$.  Call the $i$-th term
$N(i)$. Then, for each $g \in G$, $N(i)^g$ is the set of nonnegative
integers points in a certain rational polyhedron, and its elements of
degree $n \gg 0$ are counted by a quasipolynomial by \cite[Theorem
4.5.11 and Proposition 4.4.1]{Stanley97}.
\end{proof}

An immediate consequence is the following.

\begin{cor} \label{cor:ElementaryCount}
Let $S \mapsto \cE(S)/G \subseteq [k]^S/G$ be an elementary model
functor. Then $|(\cE([n])/G)/\Sym([n])|$ equals some quasipolynomial in
$n$, for all $n \gg 0$.
\end{cor}

\begin{proof}
Define a map $\cE([n]) \to \ZZ_{\geq 0}^k$ by sending the vector $\alpha$
to its {\em count vector} $\beta$, i.e., the vector in which $\beta_j$
is the number of $l \in [n]$ with $\alpha_l=j$. Note that this map is
$G$-equivariant, so the image $M_n$ is $G$-stable; and that the fibres
are precisely the $\Sym([n])$-orbits. Furthermore, the fact that $\cE$
is a model functor implies that the union $M=\bigcup_n M_n$ is downward
closed. Now apply Proposition~\ref{prop:BurnsideStanley}.
\end{proof}

\subsection{$\FIop$-schemes of product type}
\label{ssec:FiniteOverWide}

Elementary model functors are combinatorial models for the component
functor of $\FIop$-schemes of product type in the sense of
Definition~\ref{de:ProductType}, as follows.

\begin{prop} \label{prop:Elementary}
Let $L$ be a Noetherian domain and let $X$ be a width-one $\FIop$-scheme of
product type over $L$. Then the component functor $\cC_X$ is isomorphic
to an elementary model functor.
\end{prop}

Before we prove this result, we show that it holds in Example~\ref{ex:Galois}.

\begin{ex}
Let $X(S)$ be the closed subscheme of $\AA^S_{\CC(t)}$ defined by the
equations $x_i^2-t$ for all $i \in S$. We claim that this is of product
type. First, $X(S)=Z^S$ where $Z$ is the subscheme of $\AA^1_{\CC(t)}$
defined by $x^2-t$, and where the product is over $M:=\CC(t)$.  Second,
to determine the irreducible components of $X(S)$, we extend scalars to
a separable closure $\overline{M}$ of $M$, which in particular contains
a $\sqrt{t}$. Then $X_{\overline{M}}(S)$ is just $\{\pm \sqrt{t}\}^S$,
each point of which maps onto $\Spec(M)$. Thus $X$ is of
product type as claimed. The irreducible components
of $X(S)$ are orbits of irreducible components of $X_{\overline{M}}(S)$
under the Galois group, which acts diagonally on $\{\pm \sqrt{t}\}^S$
by swapping $\sqrt{t}$ and $-\sqrt{t}$. Thus $\cC_X$ is isomorphic to
the elementary model functor that maps $S$ to
$\{1,2\}^S/\Sym([2])$. 
\end{ex}

\begin{proof}
By assumption, $X(S)=Z^S$, where $Z$ is a fixed affine scheme over
$L$, and each irreducible component of $X(S)$ maps dominantly into
$\Spec(L)$.  Let $M$ be the fraction field of $L$, and let $X_M$ be the
base change of $X$ to $M$. Since each irreducible component of $X(S)$ maps
dominantly into $\Spec(L)$, basic properties of localisation imply that
the morphism $X_M(S) \to X(S)$ is a bijection at the level of irreducible
components. Furthermore, taking these bijections for all $S$, we obtain an
isomorphism $\cC_{X_M} \to \cC_X$ of contravariant functors $\FI \to \PF$.

Next let $\overline{M}$ be a separable closure of $M$ and let
$X_{\overline{M}}$ be the base change of $X_M$ to $\overline{M}$.
Then, for each $S \in \FI$, the morphism $X_{\overline{M}}(S)
\to X_M(S)$ induces a surjection $\cC_{X_{\overline{M}}}(S) \to
\cC_{X_M}(S)$, and the fibres are precisely the orbits of the Galois
group $\Gal(\overline{M}:M)$ on $\cC_{X_{\overline{M}}}(S)$ \cite[Tag
0364]{stacks-project}. In other words, $\cC_{X_M}(S)$ has a canonical
bijection to $\cC_{X_{\overline{M}}}(S)/\Gal(\overline{M}:M)$.
To complete the proof, we need to analyse the component functor of
$X_{\overline{M}}$.

To this end, let $Z_1,\ldots,Z_k$ be the irreducible
components of the base change $Z_{\overline{M}}$. Then
\[ X_{\overline{M}}(S)=Z_{\overline{M}}^S=\bigcup_{\alpha
\in [k]^S} \prod_{i \in S} Z_i^{\alpha_i} \]
where each product over $i \in S$ is a product of
irreducible varieties over the separably
closed field $\overline{M}$, and hence irreducible. To
construct our component functor, we just set
$\cE(S):=[k]^S$.

Finally, let $G$ be the image of $\Gal(\overline{M}:M)$ in $\Sym([k])$
through its action on the irreducible components $Z_1,\ldots,Z_k$
of $Z$. Then the (image of the) action of $\Gal(\overline{M}:M)$ on
irreducible components of $X_{\overline{M}}(S)$ corresponds precisely to
the (image of the) diagonal action of $G$ on $\cE(S)$, and
hence the orbit space $\cE(S)/G$ is in bijection with the
irreducible components of $X_M(S)$. This bijection, taken
for all $S$, is an isomorphism from the elementary model
functor given by $\cE(S)=[k]^S$ and the group $G \subseteq
\Sym([k])$.
\end{proof}

\begin{re}
Note that the elementary model functors coming from $\FIop$-schemes of
product type all have $\cE(S)=[k]^S$ rather than just
$\cE(S) \subseteq [k]^S$. The set of count vectors is therefore all of $\ZZ_{\geq
0}^k$. However, in our proof of the Main Theorem we will need to do
induction over the poset of downward closed subsets of $\ZZ_{\geq 0}^k$;
this requires the greater generality in the definition of elementary
model functors.
\end{re}

\begin{cor} \label{cor:FiniteType}
The Main Theorem holds for affine $\FIop$-schemes of product type over
some Noetherian domain.
\end{cor}

\begin{proof}
This is an immediate corollary of Proposition~\ref{prop:Elementary}
and Corollary~\ref{cor:ElementaryCount}.
\end{proof}

We are now in a position to prove the following result, most of which
also follows from combining results from \cite{VanLe18} (linearity of
codimension) and \cite{Nagel17} (the form of the Hilbert function).

\begin{thm} \label{thm:DimComp}
Let $X$ be an affine width-one $\FIop$-scheme $X$ of finite type over a
Noetherian ring $K$. Assume that $X([n])$ is not the empty
scheme for any $n$. Then for $n \gg 0$ the Krull dimension of $X([n])$
is eventually equal to an affine-linear polynomial in $n$, and the number
of irreducible components $|\cC_X([n])|$ is bounded from above by $c^n$
for some constant $c \geq 1$.
\end{thm}

\begin{proof}
By Lemma~\ref{lm:Reduced} we may assume that $X=\Spec(B)$ is reduced,
and by Proposition~\ref{prop:Nicefy} we may assume that $X$ is nice. Since
$X([n])$ is not the empty scheme for any $n$, we have $1 \neq 0$ in $B_0$.
By the Shift Theorem~\ref{thm:Shift} and Proposition~\ref{prop:Shift}
there exists an $S_0 \in \FI$ and a nonzero $h \in B(S_0)$ such that
$X':=(\Sh_{S_0} X)[1/h]$ is of product type; in particular, it sends $S
\to Z^S$ for some reduced scheme $Z$ of finite type over $L:=B(S_0)[1/h]$.

Let $Y$ be the closed $\FIop$-subscheme of $X$ defined by the vanishing
of $h$. For any $S \in \FI$, $X(S_0 \sqcup S)$ is the union of $Y(S_0
\sqcup S)$ and all $X(\pi) X'(S)$ where $\pi$ ranges over the finite
set $\Sym(S_0 \sqcup S)$. Therefore,
\[ \dim X(S_0 \sqcup S)=\max\{\dim(Y(S_0 \sqcup S)),\dim( X'(S))\}. \]
By Noetherian induction using Theorem~\ref{thm:Noetherianity}, we may assume that the
theorem holds for $Y$. On the other hand,
$\dim(X'(S))=\dim(L) + |S| \cdot \dim(Z)$.
We conclude that, for $n \gg 0$, $\dim(X([n]))$ is a maximum
of two affine-linear functions of $n$, hence itself an affine-linear
function of $n$. Similarly, to bound $|\cC_X(S_0 \sqcup S)|$
we claim that
\[ |\cC_X(S_0 \sqcup S)| \leq |\cC_Y(S_0 \sqcup S)| +
(|S|+|S_0|)(|S|+|S_0|-1) \cdots (|S|+1) |\cC_{X'}(S)|. \]
Indeed, if $C$ is an irreducible component of $X(S_0 \sqcup S)$, then
either $C$ is contained in $Y(S_0 \sqcup S)$ (and then a component
there) or else there exists an injection $\pi:S_0 \to S_0 \sqcup S$
such that $B(\pi)h$ is not identically zero on $C$. In the latter case,
let $\sigma \in \Sym(S_0 \sqcup S)$ be any element with $\sigma \circ
\pi = \id_{S_0}$. Then $B(\sigma)B(\pi)h=h$, and hence $C=X(\sigma)C'$
for a component $C'$ of $X(S_0 \sqcup S)$ on which $h$ is nonzero. These
components correspond bijectively to components of $X'(S)$. This explains
the second term, where the first $|S_0|$ factors count the number of
possibilities for $\pi$.

Now the first term is bounded by an exponential function of
$|S_0|+|S|$ by the
induction hypothesis, and the second term is bounded by an
exponential function by the proof of Proposition~\ref{prop:Elementary}. Hence so
is the sum.
\end{proof}

\subsection{The orbit-counting lemma for groupoids}
\label{ssec:OrbitCounting}

It turns out that in the general case of the Main Theorem, the (Galois)
group that featured in the proof of Corollary~\ref{cor:FiniteType},
is replaced by a suitable {\em groupoid}. We briefly recall
the relevant set-up.

Let $G$ be a finite groupoid, that is, a category whose class of
objects is a finite set $Q$ and in which for any $p,q \in Q$ the
set $G(p,q):=\Hom(p,q)$ is a finite set all of whose elements are
isomorphisms. Rather than homomorphisms or isomorphisms, we will call
these elements {\em arrows}.

For a groupoid to act on a finite set $X$, one first specifies an {\em
anchor map} $a:X \to Q$. For $p \in Q$, set $X_p:=a^{-1}(p)$.  Then,
an action of $G$ on $X$ consists of the data of a map $\phi(g):X_p \to
X_q$ for each homomorphism $g:p \to q$, subject to the conditions that
$\phi(\id_p)=\id_{X_p}$ and $\phi(h \circ g)=\phi(h) \circ \phi(g)$
for any two arrows $g:p \to q$ and $h:q \to r$. We often write $gx$
instead of $\phi(g)(x)$.

Write $G(p):=\bigcup_{q \in Q} G(p,q)$ for the set of arrows from
$p$. For $x \in X_p$ we have a map $G(p) \to X, g \mapsto gx$. The image
of this map is called the {\em orbit} of $x$ and denoted $G(p)x$. On
the other hand, we write $G(p,p)_x:=\{g \in G(p,p) \mid gx=x\}$,
the stabiliser of $x$ in $G(p,p)$, which is a subgroup of the group
$G(p,p)$.  The map $G(p) \to G(p)x$ yields a bijection $G(p)/G(p,p)_x
\to G(p)x$; here $G(p,p)_x$ acts freely on $G(p)$ by precomposition,
so that $|G(p)x|=|G(p)|/|G(p,p)_x|$.  Furthermore, for every element $y
\in G(p)x$ we have $|G(a(y))|=|G(p)|$ and $|G(a(y),a(y))_y|=|G(p,p)_x|$.

Finally, for $g \in G(p,p)$ we write $X_p^g$ for the set of elements
$x \in X_p$ with $gx=x$. The following is a generalisation of the
orbit-counting lemma for groups.

\begin{lm} \label{lm:Burnside}
The number of orbits of $G$ on $X$ equals
\[ \sum_{p \in Q} \frac{1}{|G(p)|} \sum_{g \in G(p,p)}
|X_p^g|. \hfill
\]
\end{lm}

\begin{proof}
We count the triples $(p,g,x)$ with $p \in Q$ and $x \in X_p$ and $g
\in G(p,p)$ with $gx=x$ and $|G(p)|=N$ in two
different ways. If we first fix $x$, then we are forced to take $p:=a(x)$,
and we obtain
\begin{align*} \sum_{x \in X:|G(a(x))|=N} |G(a(x),a(x))_x| &= \sum_{x \in
X:|G(a(x))|=N} |G(a(x))|/|G(a(x))x|\\ &= \sum_{x \in
X:|G(a(x))|=N} N/|G(a(x))x|. \end{align*}
This is $N$ times the number of orbits of $G$ on the set of
$x$ with $|G(a(x))|=N$.

On the other hand, if we first fix $p$ with $|G(p)|=N$
and $g \in G(p,p)$,
then we find
\[ \sum_{p \in Q: |G(p)|=N} \sum_{g \in G(p,p)} |X_p^g|. \]
Hence the number of orbits of $G$ on the set of $x$ with
$|G(a(x))|=N$ equals
\[ \frac{1}{N} \sum_{p \in Q: |G(p)|=N} \sum_{g \in G(p,p)} |X_p^g|. \]
Now sum over all possible values of $N$ to obtain the formula
in the lemma.
\end{proof}

\subsection{Model functors} \label{ssec:ModelFunctors}

Elementary model functors are special cases of a more general class of
functors $\FI \to \PF$, which we call {\em model functors}.
Their construction is motivated by Theorem~\ref{thm:Shift} and
Proposition~\ref{prop:Elementary}, as we will see below.

\begin{figure}
\begin{center}
\includegraphics{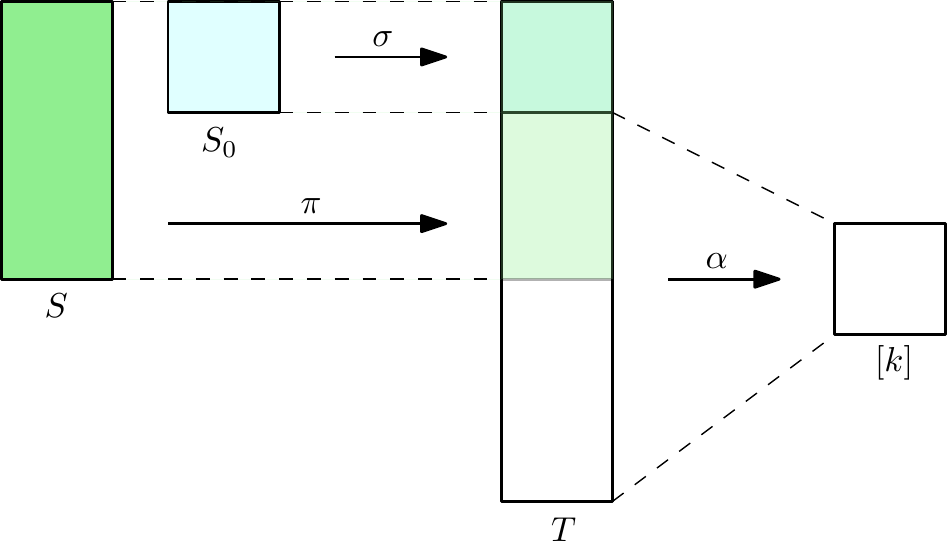}
\caption{The construction of $\cF(\pi)$ in Definition~\ref{de:Model}.}
\label{fig:diag}
\end{center}
\end{figure}

Fix $S_0 \in \FI$, a $k \in \ZZ_{\geq 0}$, and a subfunctor $\cE:\FI
\to \PF$ of the functor $S \mapsto [k]^S$, where we require that for
all $\pi \in \Hom_\FI(S,T)$, $\cE(\pi)$ is defined everywhere. Hence,
by passing to count vectors as in \S\ref{ssec:FirstQuasi}, $\cE$ is
uniquely determined by a downward closed subset of $\ZZ_{\geq 0}^k$.

Then we define a new functor $\cF: \FI \to \PF$ on objects by
\[ \cF(S)=\{(\sigma,\alpha) \mid \sigma \in
\Hom_{\FI}(S_0,S),\ \alpha \in
\cE(S \setminus \im(\sigma))\} \]
and on a morphism $\pi \in \Hom_{\FI}(S,T)$ as follows: for
$(\sigma,\alpha) \in \cF(T)$
we set
\[ \cF(\pi)((\sigma,\alpha)):=
\begin{cases}
\text{undefined if $\im(\sigma) \not \subseteq \im(\pi)$; and }\\
(\sigma', \alpha \circ
\pi|_{S \setminus \im(\sigma')}) \text{ where }
\sigma':=\pi^{-1} \circ \sigma
\text{ otherwise.}
\end{cases}
\]
Figure~\ref{fig:diag} depicts all relevant maps.
At the level of species (so remembering the maps $\cF(\pi)$ only
when $\pi$ is a bijection), this is an instance of a well-known
construction: $\cF$ is the product of the species that maps $S$
to its set of bijections $S_0 \to S$ and the species that maps $S$
to $\cE(S)$.

Next let $\sim_S$ be an equivalence relation on $\cF(S)$ for each $S$,
and assume that these relations satisfy the following three axioms:
\begin{description}
\item[Axiom (1)] if $(\sigma,\alpha) \sim_T (\sigma',\alpha')$ and $\pi \in
\Hom_{\FI}(S,T)$ has $\im(\pi) \supseteq \im(\sigma) \cup \im(\sigma')$,
then $\cF(\pi)((\sigma,\alpha)) \sim_S \cF(\pi)((\sigma',\alpha'))$;

\item[Axiom (2)] conversely, if the pairs $(\sigma,\alpha) \in \cF(T)$,
$(\sigma'',\alpha'') \in \cF(S)$, and the map 
$\pi \in \Hom_{\FI}(S,T)$ satisfy $\im(\pi) \supseteq \im(\sigma)$ and
$\cF(\pi)((\sigma,\alpha)) \sim_S (\sigma'',\alpha'')$, then there exists
a pair $(\sigma',\alpha') \in \cF(T)$
with $\im(\sigma') \subseteq \im(\pi)$ such that $(\sigma',\alpha')
\sim_T (\alpha,\sigma)$ and
$\cF(\pi)((\sigma',\alpha'))=(\sigma'',\alpha'')$; and

\item[Axiom (3)] if $(\sigma,\alpha) \sim_T (\sigma',\alpha')$ and $i,j \in T
\setminus (\im(\sigma) \cup \im(\sigma'))$, then $\alpha(i)=\alpha(j)
\Leftrightarrow \alpha'(i)=\alpha'(j)$.
\end{description}
The first axiom ensures that $\cF/\!\!\sim: S \mapsto
\cF(S)/\!\!\sim_S$ is
a functor $\FI \to \PF$ that comes with a canonical surjective
morphism $\cF \to \cF/\!\!\sim$ in the sense of Definition~\ref{de:Morphism};
in particular, this implies that $\sim_S$ is preserved under the symmetric
group $\Sym(S)$ acting on $\cF(S)$.  The second and third axioms will
be crucial in \S\ref{ssec:SecondQuasi}.

\begin{de} \label{de:Model}
A functor of the form $S \mapsto \cF(S)/\!\!\sim_S$ as constructed above is
called a {\em model functor}. 
\end{de}

\begin{re}
Each elementary model functor $S \mapsto \cE(S)/G$ is isomorphic to
a model functor with $S_0=\emptyset$ (so that we may leave out the
$\sigma$s from the pairs) and $\alpha \sim_S \alpha'$ if and only if
$\alpha' \in G \alpha$. We will see that, conversely, a model functor
gives rise to to certain groupoids that play the role of
$G$.
\end{re}

We revisit Example~\ref{ex:cube} from the perspective of
model functors.

\begin{ex} \label{ex:Model}
Let $X(S)$ be the subscheme of $\AA^S_{\CC}$ defined by the equations
$x_j^d - x_l^d$ for all $j,l \in S$. Set $\zeta:=e^{2\pi i/d}$. If we fix a
$j_0 \in S$, there is a bijection between irreducible components of $X(S)$
and elements of $(\ZZ/d\ZZ)^{S \setminus \{j_0\}}$, where the component
corresponding to $\alpha \in (\ZZ/d\ZZ)^{S \setminus \{j_0\}}$ is that
in which each $x_j, j \neq j_0$, equals $\zeta^{\alpha_j}
x_{j_0}$. By identifying $\ZZ/d\ZZ$ with $[d]$ in
the natural manner and regarding $j_0$ as the image of a $\sigma \in
\Hom_\FI(S_0,S)$ where $S_0$ is a singleton, we obtain a surjection from the
functor
\[
\cF:S \mapsto \{(\sigma,\alpha) \mid \sigma \in
\Hom_{\FI}(S_0,S), \alpha \in [d]^{S \setminus
\im(\sigma)}\}
\]
to the component functor $\cC_X$. A pair $(\sigma,\alpha)$
is mapped to the same component as a pair $(\sigma',\alpha')$
if and only if either $(\sigma,\alpha)=(\sigma',\alpha')$ or else $\sigma,\sigma'$
have distinct images $j_0,j_0'$ and for all $j \in S \setminus
\{j_0,j_0'\}$ we have $\alpha'(j)-\alpha'(j_0)=\alpha(j)$ and
$\alpha(j)-\alpha(j'_0)=\alpha'(j)$ (both in $\ZZ/d\ZZ$). This defines an equivalence relation
$\sim_S$ satisfying Axioms (1)--(3), and $\cC_X$ is isomorphic to the
model functor $S \mapsto \cF(S)/\!\!\sim_S$. 
\end{ex}

\begin{re}
Informally, we think of $(\sigma,\alpha)$ as a word in $[k]^S$ in which
the letters corresponding to $\im(\sigma) \subseteq S$ are concealed. If
$(\sigma', \alpha') \sim_S (\sigma, \alpha)$, then $(\sigma',\alpha')$
corresponds to a different word in which the letters corresponding to
$\im(\sigma')$ are concealed. Outside of $\im(\sigma) \cup \im(\sigma')$,
by Axiom (3), the two words are equal up to some permutation of
$[k]$. Axioms $(1)$ and $(2)$ simply ask that these equivalent words
behave well with respect to $\FI$-morphisms. In the next section, we
will roughly speaking attempt to ``discover'' the information concealed
in $\im(\sigma)$ by looking at equivalent pairs $(\sigma',\alpha')$
where $\im(\sigma')$ does not contain $\im(\sigma)$.
\end{re}

\begin{ex} \label{ex:Codes}
Fix a finite field $\FF_q$ and a natural number $m$. Let $S$ be a
finite set. On the set of rank-$m$ matrices in $\FF_q^{m \times S}$
act four groups:
\begin{enumerate}
\item $\GL_m(\FF_q)$ by row operations;
\item $\Sym(S)$ by permuting columns;
\item $(\FF_q^*)^S$ by scaling columns; and
\item $\Aut(\FF_q)$ by acting on all coordinates.
\end{enumerate}

Modding out only $\GL_m(\FF_q)$, we obtain the set of $m$-dimensional
codes in $\FF_q^S$, and two codes are called isomorphic if they are in
the same orbit under $\Sym(S) \ltimes G(S)$, where $G(S)=\Aut(\FF_q)
\ltimes (\FF_q^*)^S$.

Let $\cN$ be a functor $\FI \to \PF$ that assigns
to $S$ a set of $m$-dimensional codes in $\FF_q^S$ and to an injective
map $\pi:S \to T$ the map that sends a code $C \in \cN(T)$ to the code
\[ \{v \circ \pi \mid v \in C\} \subseteq \FF_q^S, \]
provided that this linear space still has dimension $m$. In particular,
we require that the code above is then an element of $\cN(S)$. This
implies that $\cN(S)$ is preserved under $\Sym(S)$ and that $\cN$
is closed under puncturing. Furthermore, we assume that $\cN(S)$ is
preserved under $G(S)$. To prove Theorem~\ref{thm:Codes}, we
need to
count the orbits of $\Sym(S) \ltimes G(S)$ on $\cN(S)$; this
is the same as the number of orbits of $\Sym(S)$ on
$\cM(S):=\cN(S)/G(S)$. We will informally call the elements
of $\cM(S)$ codes, as well. This $\cM$ is another functor $\FI
\to \PF$, and we claim that it is isomorphic to a model
functor.

To see this, let $k:=|\FF_q^m/\FF_q^*|=1+(q^m-1)/(q-1)$ and fix any
bijection between $[k]$ and $\FF_q^m/\FF_q^*$ by which we identify
these two sets.  Also set $S_0:=[m]$. Now define
\begin{align*} \cE(S):=\{&\alpha \in
((\FF_q)^m/\FF_q^*)^S \mid \\ &\text{the row space of the matrix $(I|\alpha)
\in (\FF_q)^{[m] \times (S_0 \sqcup S)}$ is in $\cN(S_0 \sqcup
S)$}\}, \end{align*}
where $(I|\alpha)$ is the $[m] \times (S_0 \sqcup S)$-matrix which in
the $[m] \times S_0$-block has the identity matrix $I$ and in the $[m]
\times S$-block has the matrix $\alpha$. Furthermore,
define $\cF(S)$ as the set of pairs $(\sigma,\alpha)$ with $\sigma \in
\Hom_{\FI}(S_0,S)$ and $\alpha \in \cE(S \setminus
\im(\sigma))$.  We have
a surjective morphism $\Psi:\cF \to \cM$ in the sense of
Definition~\ref{de:Morphism} that sends $(\sigma,\alpha)$
to the orbit under $G(S)$ of the row space of $(I|\alpha)
\in \FF_q^{m \times S}$, where now $I
\in \FF_q^{[m] \times \im(\sigma)}$ is the permutation matrix with entries
$\delta_{i,\sigma^{-1}(j)}$.

We define $\sim_S$ on $\cF(S)$ by $(\sigma,\alpha) \sim (\sigma',\alpha')$
if and only if $\Psi((\sigma,\alpha))=\Psi((\sigma',\alpha'))$. We
need to show that this satisfies Axioms (1),(2), and (3). Axiom
(1) follows from the fact that if $(\sigma,\alpha),(\sigma',\alpha')$
represent the same ($G(T)$-orbit of) code(s), then the same is true
after puncturing this code in a coordinate outside $\im(\sigma) \cup
\im(\sigma')$.  To see Axiom (2), we puncture a code
$C$ represented by $(\sigma,\alpha)$ in a coordinate $i \in T \setminus
\im(\sigma)$, and assume that the resulting code $C'$,
represented by $(\sigma,\alpha|_{T \setminus \{i\}})$, is
also represented by some other pair
$(\sigma'',\alpha'') \in \cF(T \setminus \{i\})$. Then the projection
of $C'$ to $\FF_q^{\im(\sigma'')}$ is surjective, and hence the same
holds for $C$, so that $C$ is also represented by a pair of the form
$(\sigma'',\alpha')$. Finally, Axiom (3) says that if in the generator
matrix $(I|\alpha)$ for the code $C$ represented by $(\sigma,\alpha)$
the columns labelled $i,j \in T \setminus \im(\sigma)$ are parallel, then
the same holds for any other pair $(\sigma',\alpha')$ that represents $C$
and satisfies $i,j \in T \setminus \im(\sigma')$---the main
point here is that parallelness is preserved under field automorphisms.
\end{ex}

\subsection{A second quasi-polynomial count} \label{ssec:SecondQuasi}

We use the notation from \S\ref{ssec:ModelFunctors}. So $S \mapsto
\cF(S)/\!\!\sim_S$ is a model functor, where $\cF(S)$ is the set of pairs
$(\sigma,\alpha)$ with $\sigma \in \Hom_\FI(S_0,S)$ and $\alpha \in \cE(S
\setminus S_0) \subseteq [k]^{S \setminus S_0}$, and where $\sim_S$ is
an equivalence relation on $\cF(S)$ that satisfies Axioms (1),(2),(3)
for model functors. The following theorem and its proof are
elementary, but quite subtle---indeed, this is probably the
most intricate part of the paper.

\begin{thm} \label{thm:ModelFunctor}
Let $S \mapsto \cF(S)/\!\!\sim_S$ be a model functor $\FI \to
\PF$. Then there exists a quasipolynomial $f$ such that the number of
$\Sym([n])$-orbits on $\cF([n])/\!\!\sim_{[n]}$ equals $f(n)$ for all
$n \gg 0$.
\end{thm}

This immediately implies Theorem~\ref{thm:Codes}.

\begin{proof}[Proof of Theorem~\ref{thm:Codes}]
By Example~\ref{ex:Codes}, the number of length-$n$ elements
in $\cC$ is the number of $\Sym([n])$-orbits on
$\cF([n])/\!\!\sim_{[n]}$ for a suitable model functor. The
result now follows from Theorem~\ref{thm:ModelFunctor}.
\end{proof}

To prove Theorem~\ref{thm:ModelFunctor}, we introduce the notion of
sub-model functor.

\begin{de}
Suppose that we have, for each $S \in \FI$, a subset $\cE'(S) \subseteq
\cE(S)$ such that, first, for all $\pi \in \Hom_{\FI(S,T)}$ the pull-back
map $[k]^T \to [k]^S$ that maps $\cE(T)$ into $\cE(S)$ also maps
$\cE'(T)$ into $\cE'(S)$; and second, for all $(\sigma,\alpha) \sim_S
(\sigma',\alpha') \in \cF(S)$ with $\alpha \in \cE'(S
\setminus \im(\sigma))$, we also have $\alpha'
\in \cE'(S \setminus \im(\sigma'))$. Then $S \mapsto \cF'(S)/\!\!\sim_S$, where
\[ \cF'(S):=\{(\sigma,\alpha) \in \cF(S) \mid \alpha \in
\cE'(S)\} \]
and where $\sim_S$ stands for the restriction of $\sim_S$ to $\cF'(S)$
is a model functor called a {\em sub-model functor} of
$\cF$.
\end{de}

\begin{proof}[Proof of Theorem~\ref{thm:ModelFunctor}]
The proof of this theorem will take up the remainder of this
subsection. This will involve an induction hypothesis for a
sub-model functor $\cF'$ of $\cF$ and the construction of
a certain groupoid for the complement $\cF \setminus \cF'$.

Let $M \subseteq \ZZ_{\geq 0}^k$ be the downward-closed set consisting
of all the count vectors of elements in $\cE(S)$ for $S$ running over
$\FI$. Since, by Dickson's lemma, the set of downward-closed sets in
$\ZZ_{\geq 0}^k$ satisfies the descending chain property, we may assume
that the theorem holds for all model functors whose corresponding downward
set is strictly contained in $M$.

Our goal is now to find $\cE'(S) \subseteq \cE(S)$ that defines a
sub-model functor $\cF'$ of $\cF$ such that the downward
closed set $M'$ of $\cE'$ is strictly contained in $M$.
Then we have
\[ |\cF(S)/\!\!\sim_S| = |\cF'(S)/\!\!\sim_S| \  + \  |(\cF(S)
\setminus \cF'(S))/\!\!\sim_S|
\]
and this equality continues to hold if we mod out the action of $\Sym(S)$
on the three sets in question. Hence by the induction hypothesis
we are done if we can show that the number
of $\Sym(S)$-orbits on $(\cF(S) \setminus \cF'(S))/\!\!\sim_S$ grows
quasipolynomially in $|S|$ for $|S| \gg 0$. In this induction argument,
we may of course assume that $M$ is not empty---otherwise, the
quasipolynomial $0$ will do.

To construct $\cE'$ we proceed as follows. Let
\[ A:=\{I \subseteq [k] \mid \exists v \in M: v+\ZZ_{\geq 0}^I
\subseteq M\}; \]
here, and in the rest of the paper, we identify $\ZZ_{\geq 0}^I$ with $\ZZ_{\geq 0}^I \times \{0\}^{[k]
\setminus I} \subseteq \ZZ_{\geq 0}^k$.
Note that $A$ is nonempty because $M$ is. Let $I$ be an inclusion-wise
maximal element of $A$ and set
\[ d:=\max\left\{\sum_{l \in [k] \setminus I} v(l) \mid v +
\ZZ_{\geq 0}^I \subseteq M\right\}. \]
This is well-defined, since if the sum of the entries in such $v$
at positions outside $I$ were unbounded, then $I$ would be contained in a
strictly larger element of $A$. Choose $v \in M$ such that $v+\ZZ_{\geq
0}^I \subseteq M$ and $\sum_{l \in [k] \setminus I} v(l)=d$.

\begin{lm}
For $j \in \ZZ_{\geq 0}$, define $v_j:=v+ j \cdot \left(\sum_{i \in
I}e_j\right) \in M$. There exists a $j \in \ZZ_{\geq 0}$ such that
the vectors in $M$ that are componentwise $\geq v_j$ are precisely the
vectors in $v_j+\ZZ_{\geq 0}^I$; this then also holds for all larger
values of $j$.
\end{lm}

\begin{proof}
Suppose that for every $j \in \ZZ_{\geq
0}$ there is a $w_j  \in M \setminus (v_j +\ZZ_{\geq 0}^I)$ that is
componentwise $\geq v_j$. Then by Dickson's lemma the sequence $w_1|_{[k]
\setminus I},w_2|_{[k] \setminus I},w_3|_{[k]\setminus I},\ldots $ would
contain an infinite subsequence, labelled by $i_1<i_2<\ldots$, that
weakly increases componentwise. It follows that $w_{i_1} + \ZZ_{\geq
0}^I \subseteq M$ because $M$ is downward closed and the
entries of $w_{i_j}$ labelled by $I$ diverge to infinity. Furthermore,
by the choice of $w_{i_1}$,
the sum $\sum_{l \in [k] \setminus I} w_{i_1}(l)$ is strictly larger
than $d$, a contradiction.
\end{proof}

From now on, we will make the following assumption on $v$:

\begin{quote}
{\em The entries $v(l)$ are $\gg 0$ for all $l \in I$.}
\end{quote}

In particular, after replacing $v$ with the $v_j$ from the lemma,
this implies that the vectors in $M$ that are componentwise $\geq v$
are precisely the vectors in $v+\ZZ_{\geq 0}^I$. In the course of our
reasoning, we will need further assumptions on how large the $v(l)$
with $l \in I$ are---to avoid technicalities, however, we make no attempt to
specify a precise lower bound that works.

The set $I$ is now uniquely determined by $v$ as the set of all positions
$l \in [k]$ where $v(l)$ is very large. We call it the {\em frequent set}
of $v$.

Write $v=v_0 + v_1$ with $v_0 \in \ZZ_{\geq 0}^{[k] \setminus I}$ and
$v_1 \in \ZZ_{\geq 0}^I$.  We now define $\cF':\FI \to \PF$ by setting
$\cF'(S)$ to be the set of pairs $(\sigma,\alpha) \in \cF(S)$ for which
there exists {\em no} pair $(\sigma',\alpha') \sim_S (\sigma,\alpha)$
such that the count vector of $\alpha'$ is in $\tilde{v} + \ZZ_{\geq
0}^I$, where $\tilde{v}:=v+v_1=v_0 + 2 v_1$. Notice the factor $2$; the
relevance of this will become clear towards the end of the proof.

\begin{lm}
The association $S \mapsto \cF'(S)/\!\!\sim_S$ is a sub-model functor of
$\cF$.
\end{lm}

\begin{proof}
By definition, $\cF'(S)$ is a union of of $\sim_S$-equivalence classes,
and uniquely determined by a subset $\cE'(S) \subseteq \cE(S)$ of allowed
second components $\alpha$. So we need only prove that $\cF'$ is
preserved under morphisms.

Hence let $(\sigma,\alpha) \in \cF'(T)$, let $\pi \in \Hom_\FI(S,T)$
satisfy $\im(\pi) \supseteq \im(\sigma)$, and consider
$(\tilde{\sigma},\tilde{\alpha}):=\cF(\pi)((\sigma,\alpha))$. If
$(\tilde{\sigma},\tilde{\alpha}) \sim_S (\sigma'',\alpha'')$
where $\alpha''$ has a count vector in $\tilde{v}+\ZZ_{\geq
0}^I$, then by Axiom (2) for model functors there exists
a pair $(\sigma',\alpha') \sim_T (\sigma,\alpha)$ with
$\cF(\pi)((\sigma',\alpha'))=(\sigma'',\alpha'')$. This means that the
count vector of $\alpha'$ is an element of $M$ that is componentwise
greater than or equal to the count vector of $\alpha''$ and hence, by the
choice of $\tilde{v}$, the count vector of $\alpha'$ lies in $\tilde{v}+\ZZ_{\geq 0}^I$. This
contradicts the fact that $(\sigma,\alpha) \in \cF'(T)$. Therefore,
$\cF(\pi)((\sigma,\alpha)) \in \cF'(S)$.
\end{proof}

Furthermore, we observe that the downward closed subset $M'$
corresponding to $\cE'$ is strictly contained in $M$, since it does not
contain $\tilde{v} \in M$. Hence the induction hypothesis applies to $\cF'$.

Our task is therefore reduced to counting the $\Sym(S)$-orbits on the
set of $\sim_S$-equivalence classes on $\tilde{\cF}(S):=\cF(S) \setminus
\cF'(S)$. Note that $\tilde{\cF}$ is not a sub-model functor of $\cF$;
rather, $\tilde{\cF}(S)$ consists of all pairs $(\sigma,\alpha)$ that are
equivalent to some pair $(\sigma',\alpha')$ where $\alpha'$ has a count
vector in $\tilde{v}+\ZZ_{\geq 0}^I$. Before counting these, we will
work for a while with the larger set $\cF''(S) \supseteq \tilde{\cF}(S)$
consisting of all pairs $(\sigma,\alpha) \in \cF(S)$ that are equivalent
to some pair $(\sigma',\alpha')$ where the count vector of $\alpha'$
is in $v+\ZZ_{\geq 0}^I \supseteq \tilde{v} + \ZZ_{\geq 0}^I$.

For the time being, fix a pair $(\sigma,\alpha)  \in \cF''(T)$ where
$\alpha$ has count vector in $v + \ZZ_{\geq 0}^I$. This implies that all
elements in $I$ occur very frequently among the entries of $\alpha$, while
all elements in $[k] \setminus I$ occur very infrequently; we call $I$
the {\em frequent set} of $\alpha$ and of the pair $(\sigma,\alpha)$. Now
let $(\sigma',\alpha') \sim_T (\sigma,\alpha)$. Since, by Axiom (3)
for model functors, the equality patterns of $\alpha'$ and $\alpha$
agree on $T \setminus (\im(\sigma) \cup \im(\sigma'))$, there exists a
set $I' \subseteq [k]$ of the same cardinality as $I$, and a bijection
$g=g((\sigma',\alpha'),(\sigma,\alpha)):I \to I'$ such that, for $i
\in T \setminus (\im(\sigma) \cup \im(\sigma'))$, we have $\alpha(i)
= l \in I$ if and only if $\alpha'(i)=g(l)$. In particular, $\alpha'$,
too, has a distinguished set $I' \subseteq [k]$ of elements that occur
very frequently among its entries, while the complement occurs very
infrequently. We call $I'$ the frequent set of $\alpha'$.

Furthermore, if also $(\sigma'',\alpha'') \sim_T
(\sigma,\alpha)$, then we have
\begin{equation} \label{eq:Groupoid1} g((\sigma'',\alpha''),(\sigma',\alpha')) \circ
g((\sigma',\alpha'),(\sigma,\alpha))=g((\sigma'',\alpha''),(\sigma,\alpha))
\end{equation}
as a map from $I$ to the frequent set $I''$ of $\alpha''$, and we have
\begin{equation} \label{eq:Groupoid2}
g((\sigma,\alpha),(\sigma,\alpha))=\id_{I}. \end{equation}

Still using elements from the $\sim_T$-equivalence class of $(\sigma,\alpha)$,
we define a relation $\equiv$ on $T$ as follows: first, $\equiv$
is reflexive, and second, for $i \neq j$ we have $i \equiv j$ if
and only if there exists $(\sigma',\alpha') \sim_T (\sigma,\alpha)$
with $i,j \not \in \im(\sigma')$ and $\alpha'(i)=\alpha'(j) \in
I':=g((\alpha',\sigma'),(\alpha,\sigma))I$.

\begin{lm} \label{lm:Transposition}
Assume that $i \equiv j$. Then $(\sigma,\alpha) \sim_T \cF((i\ j))
(\sigma,\alpha)$, where $(i\ j)$ is the transposition of $i$
and $j$.
\end{lm}

\begin{proof}
If $i=j$, then the statement is obvious. Otherwise, there exists a
pair $(\sigma',\alpha') \sim_T (\sigma,\alpha)$ such that $\alpha'$ is
defined at $i$ and $j$ and takes the same value $l$ in the frequent set of
$\alpha'$. Then $\cF((i\ j)) (\sigma',\alpha') = (\sigma',\alpha')$ and by
Axiom (1), $\cF((i\ j))(\sigma,\alpha) \sim_T \cF((i\ j))(\sigma',\alpha')
= (\sigma',\alpha') \sim_T (\sigma,\alpha)$, as desired.
\end{proof}

\begin{lm}
The relation $\equiv$ is an equivalence relation on $T$.
\end{lm}

\begin{proof}
First note that---using Axiom (3) for model functors---if $i \neq j$
satisfy $i \equiv j$, then in fact for {\em all} $(\sigma',\alpha')
\sim_T (\sigma,\alpha)$ with $i,j \in T \setminus \im(\sigma')$ we
have $\alpha'(i)=\alpha'(j) \in I'$.  Since $\equiv$ is reflexive and
symmetric by definition, we only need to show transitivity. For this, assume
that $i \equiv j$ and $j \equiv h$, where we may assume that $i,j,h$
are all distinct. Let $(\sigma',\alpha') \sim_T (\sigma,\alpha)$ be
such that $\alpha'(j)=\alpha'(h) \in I'$. Now if $\alpha'$ is defined
at $i$, then $i \equiv j$ implies that also $\alpha'(i)=\alpha'(j)$
so that $i \equiv h$. Assume that $\alpha'$ is not defined at $i$. Then
let $i' \in T\setminus \{i,j,h\}$ be a position where $\alpha'$ {\em is}
defined and such that $i' \equiv i$---this exists, because
there exists a pair $(\sigma'',\alpha'') \sim_T
(\sigma,\alpha)$ for which $\alpha''$ is defined
at $i$ (and defined and equal at $j$); now set
$l:=\alpha''(i)$, an element in the frequent set of
$\alpha''$, and take for $i'$ any element from
$((\alpha'')^{-1}(l)) \setminus \{i,j,h\}$.
Then by
Lemma~\ref{lm:Transposition} we have $(\sigma',\alpha') \sim_T \cF((i\
i'))(\sigma',\alpha')$ and the latter element {\em is} defined at $i,j,h$.
This proves transitivity.
\end{proof}

So for all elements $(\sigma',\alpha')$ in the $\sim_T$-equivalence
class of $(\sigma,\alpha)$ we have the same, well-defined equivalence
relation $\equiv$ on $T$. Let $T_1 \subseteq T$ be the set of elements
that form a singleton class; we call $T_1$ the {\em core} of (the
$\sim_T$-equivalence class of) $(\sigma,\alpha)$. Note that $T_1=T_{10}
\sqcup T_{11}$ where $T_{10}$ is the set of those positions in $T$
that are in $\im(\sigma')$ for {\em all} $(\sigma',\alpha') \sim_T
(\sigma,\alpha)$ and $T_{11}$ is the set of elements that are in
$(\alpha')^{-1}([k] \setminus I')$ for some $(\sigma',\alpha') \sim_T
(\sigma,\alpha)$ and $I'=g((\sigma',\alpha'),(\sigma,\alpha))I$. We have
$T_{10} \subseteq \im(\sigma)$ and also $T_{11} \subseteq \im(\sigma)
\cup \alpha^{-1}([k] \setminus I)$; in particular, $|T_1|$ is bounded
from above by $|S_0| + |\alpha^{-1}([k] \setminus I)|=|S_0|+d$, where $d$
was the number used in the construction of $v \in M$.

The same reasoning applies to {\em any} element $(\sigma,\alpha)
\in \cF''(T)$: it unambiguously determines a subset $J \subseteq
[k]$ (the frequent set of $\alpha$) of cardinality $|J|=|I|$
and a subset $T_1 \subseteq T$ of some bounded size (the core
of the pair), as well as a surjection $\tau:T \setminus T_1
\to J$ defined by $\tau(i)=l$ if and only if there exists a
$(\alpha',\sigma') \sim_T (\alpha,\sigma)$ with $\alpha'$ defined at
$i$ and $\alpha'(i)=g((\alpha',\sigma'),(\alpha,\sigma))(l)$; and this
surjection only has large fibres. Furthermore, passing to another element
of the $\sim_T$-equivalence class, the core $T_1$ remains the same, $J$
is acted upon by a bijection $g$ to yield a $J'$, and $\tau$ is composed with
that same bijection. We now determine how certain morphisms transform
the data $J,T_1,\tau$.

\begin{lm} \label{lm:CoreMorphism}
Let $\pi \in \Hom_{\FI}(S,T)$ be such that $\im(\pi)$
contains the core $T_1 \subseteq T$ of (the
$\sim_T$-equivalence class of) $(\sigma,\alpha)$, and assume that
$(\tilde{\sigma},\tilde{\alpha}):=\cF(\pi)((\sigma,\alpha))$ lies in
$\cF''(S)$. Then the frequent set of $\tilde{\alpha}$ equals
the frequent set $J$ of $\alpha$, the core $S_1 \subseteq S$
of $(\tilde{\sigma},\tilde{\alpha})$ equals $\pi^{-1}(T_1)$,
and the surjection $\tilde{\tau}: S \setminus S_1 \to J$
determined by $(\tilde{\sigma},\tilde{\alpha})$ is the map
$\tau \circ (\pi|_{S \setminus S_1})$ where $\tau:T \setminus T_1 \to
J$ is the surjection determined by $(\sigma,\alpha)$.
\end{lm}

\begin{proof}
That the frequent set $J$ remains unchanged is immediate:
the elements that appear frequently in $\tilde{\alpha}$ also
appear frequently in $\alpha$ and vice versa.

For the statement about the core, it suffices to show that distinct
$i,j \in S$ satisfy $i \equiv j$ in the equivalence relation on $S$ defined by
$(\tilde{\sigma},\tilde{\alpha})$ if and only if $\pi(i)\equiv \pi(j)$
in the equivalence relation on $T$ defined by $(\sigma,\alpha)$.

Let $i,j \in S$ be distinct and assume $i \equiv j$, so that there exists a pair
$(\tilde{\sigma}',\tilde{\alpha}') \sim_S (\tilde{\sigma},\tilde{\alpha})$
with $\tilde{\alpha}'(i)=\tilde{\alpha}'(j) \in J$. By Axiom (2)
there exists a pair $(\sigma',\alpha') \sim_T (\sigma,\alpha)$ with
$\cF(\pi)((\sigma',\alpha'))=(\tilde{\sigma}',\tilde{\alpha}')$, and we
find that $\alpha'(\pi(i))=\alpha'(\pi(j)) \in J$, so $\pi(i) \equiv
\pi(j)$.

Conversely, let $i,j \in S$ be distinct and assume that $\pi(i) \equiv
\pi(j)$, so there exists a pair $(\sigma',\alpha') \sim_T (\sigma,\alpha)$
with $\alpha'(\pi(i))=\alpha'(\pi(j)) \in J$. Since $\im(\pi)$ contains
$T_1$, it contains all elements of $\im(\sigma') \cap T_1$. Using
Lemma~\ref{lm:Transposition} we may apply transpositions $\cF((h\
h'))$ to $(\sigma',\alpha')$ for all $h \in \im(\sigma) \setminus
\im(\pi) \subseteq T \setminus T_1$, where the $h' \equiv h$ are all
chosen distinct, disjoint from $\im(\sigma)$ and from
$\{i,j\}$, and inside $\im(\pi) \setminus T_1$,
to arrive at a $(\sigma'',\alpha'') \sim_T (\alpha',\alpha')$ still
satisfying $\alpha''(\pi(i))=\alpha''(\pi(j))$ and now also satisfying
$\im(\sigma'') \subseteq \im(\pi)$. So we may apply $\cF(\pi)$ to
$(\sigma'',\alpha'')$ and find
that $i \equiv j$.

For the statement about $\tilde{\tau}$ let $i \in S
\setminus S_1 = S \setminus \pi^{-1}(T_1)$. Then there exists a
$(\tilde{\sigma}',\tilde{\alpha}') \sim_S (\tilde{\sigma},\tilde{\alpha})$
such that $\tilde{\alpha}'$ is defined at $i$. By Axiom (2)
there exists a $(\sigma',\alpha') \sim_T (\sigma,\alpha)$ with
$\cF(\pi)((\sigma',\alpha'))=(\tilde{\sigma}',\tilde{\alpha}')$. In
particular, $\alpha'$ is defined at $\pi(i)$ and we have
$\alpha'(\pi(i))=\tilde{\alpha}'(i)$, so that
$\tilde{\tau}(i)=\tau(\pi(i))$, as desired.
\end{proof}

A special case of the lemma is that where $S=T$, and we find that
$\Sym(S)$ acts in the expected manner on the data consisting of the
frequent set (namely, trivially) and on the core $S_1$ and the map $\tau$.
The cardinality of the core is an invariant under this action, and also
preserved under the more general morphisms of Lemma~\ref{lm:CoreMorphism}.

The core is a finite subset of cardinality at most $|S_0|+d$. For
each $e \in \{0,\ldots,|S_0|+d\}$ let $\cF''_e(S)$ be the set of
elements in $\cF''(S)$ with a core of cardinality $e$, and set
$\tilde{\cF}_e(S):=\tilde{\cF}(S) \cap \cF''_e(S)$. We are done
once we establish that for each $e$ the set of $\Sym(S)$-orbits on
$\tilde{\cF}_e(S)/\!\!\sim_S$ is quasipolynomial in $|S|$ for $|S| \gg 0$.

We will decouple the core from the rest of $S$. A first justification
for this is the following lemma.

\begin{lm} \label{lm:SplitSym}
The number of $\Sym([e] \sqcup S)$-orbits on $A:=\tilde{\cF}_e([e]
\sqcup S)/\!\!\sim_{[e]\sqcup S}$ equals the number of $\Sym([e]) \times
\Sym(S)$-orbits on the set $B$ of elements in $\tilde{\cF}_e([e]
\sqcup S)/\!\!\sim_{[e]\sqcup S}$ with core equal to $[e]$.
\end{lm}

\begin{proof}
The inclusion map $B \to A$ induces a map $B/(\Sym([e]) \times \Sym(S))
\to A/\Sym([e] \sqcup S)$. This map is surjective since the core of
any element in $A$ can be moved into $[e]$ by an element of $\Sym([e]
\sqcup S)$. To see that it is also injective, note that if $\pi \in
\Sym([e] \sqcup S)$ and $(\sigma,\alpha),(\sigma',\alpha') \in
B$ satisfy
$\cF(\pi)((\sigma,\alpha))=(\sigma',\alpha')$, then $\pi$ must preserve
the common core $[e]$ of both tuples, hence $\pi \in \Sym([e]) \times
\Sym(S)$.
\end{proof}

Consider a pair $(\sigma,\alpha) \in \cF''_e([e] \sqcup T)$ with
core equal to $[e]$. To such a pair we associate the quintuple
$(J,\tau,\sigma_0,\sigma_1,\overline{\alpha})$ determined by:
\begin{enumerate}
\item the frequent set $J$ of $\alpha$;
\item the surjection $\tau: T \to J$;
\item the partially defined map $\sigma_0:S_0 \to [e]$, which is the
restriction of $\sigma$ to $\sigma^{-1}([e])$;
\item the map $\sigma_1:S_0 \setminus \dom(\sigma_0)
\to J$ defined by $\sigma_1(i)=\tau(\sigma(i))$.
\item the restriction $\overline{\alpha}$ of $\alpha$ to $[e]
\setminus \im(\sigma)$, which takes values in $[k] \setminus J$.
\end{enumerate}
The quintuple remembers everything about the pair $(\sigma,\alpha)$
{\em except} the exact values (in $T$) of $\sigma$ on $S_0 \setminus
\dom(\sigma_0)$: of these values, only their equivalence classes under
$\equiv$ classes are remembered;
these can be read off from $\sigma_1$ (and $\tau$). If another pair
$(\sigma',\alpha') \in \cF''_e([e] \sqcup T)$ with core equal to
$[e]$ yields the same quintuple, then $(\sigma',\alpha')$ differs
from $(\sigma,\alpha)$ by a permutation of $T$ that permutes elements
within their equivalence classes under $\equiv$. Hence, by repeatedly applying
Lemma\ref{lm:Transposition}, we find that $(\sigma',\alpha') \sim_{[e]
\sqcup T} (\sigma,\alpha)$.

We will use the notation $\sim_{[e] \sqcup T}$
for the induced equivalence relation on such quintuples. Note that
$\Sym([e])$ fixes ($J$ and) $\tau$, whereas $\Sym(T)$ fixes ($J$ and)
$\sigma_0,\sigma_1,\overline{\alpha}$. Also note that all components of
the quintuple except $\tau$ can take only finitely many different
values as $T$ varies. We call $(J,\sigma_0,\sigma_1,\overline{\alpha})$ the
quadruple determined by the quintuple (and {\em a fortiori} also determined by the pair
$(\sigma,\alpha)$).

We now come to the central tool for proving quasipolynomiality;
note that now we work with $\tilde{\cF}$ instead of $\cF''$---recall
that $\tilde{\cF}$ and its complement $\cF'$ were defined using
$\tilde{v}=v+v_1=v_0 + 2v_1$, while $\cF'' \supseteq \tilde{\cF}$ was
defined using $v$.

\begin{de} \label{de:GroupoidG}
Let $G$ be the finite directed graphs whose vertices are all
quadruples of pairs in $\tilde{\cF}_e([e] \sqcup T)$ with
core $[e]$, where $T$ runs through $\FI$, and whose arrows
from one quadruple $(J,\sigma_0,\sigma_1,\overline{\alpha})$ to
$(J',\sigma_0',\sigma_1',\overline{\alpha'})$ are all bijections
$g((\sigma',\alpha'),(\sigma,\alpha)):J \to J'$ coming from
pairs $(\sigma',\alpha') \sim_{[e] \sqcup T} (\sigma,\alpha)$ in
$\tilde{\cF}_e([e] \sqcup T)$ with core $[e]$ and with the prescribed
quadruples.
\end{de}

Typically, the same bijection $g$ will arise from more than one pair of
pairs with the prescribed quadruples; it then only appears once as an
arrow $g$ between those quadruples. The following proposition will be
used below to establish that $G$ is, in fact, a groupoid.

\begin{prop} \label{prop:MatchSources}
Let $(J,\tau,\sigma_0,\sigma_1,\overline{\alpha})$ be the quintuple of a
pair $(\sigma,\alpha)\in \tilde{\cF}_e([e] \sqcup T)$ with core $[e]$ and
let $g$ be an arrow in $G$ from $(J,\sigma_0,\sigma_1,\overline{\alpha})$
to $(J',\sigma_0',\sigma_1',\overline{\alpha'})$. Then $(J',g \circ
\tau,\sigma_0',\sigma_1',\overline{\alpha'})$ is also the quintuple
of some pair in $\tilde{\cF}_e([e] \sqcup T)$ with core $[e]$, and
that quintuple is $\sim_{[e] \sqcup T}$-equivalent to the original
quintuple. Furthermore, all quintuples equivalent to the original
quintuple arise in this manner.
\end{prop}

\begin{proof}
The last statement is immediate: such an equivalent quintuple comes from
an equivalent pair $(\sigma',\alpha') \sim_{[e] \sqcup T} (\sigma,\alpha)$,
and for the arrow we can take $g=g((\sigma',\alpha'),(\sigma,\alpha))$.

For the first statement, let $(\tilde{\sigma},\tilde{\alpha})
\sim_{[e] \sqcup S} (\tilde{\sigma}',\tilde{\alpha}')$
be pairs with the given quadruples such that
$g=g((\tilde{\sigma}',\tilde{\alpha}'),(\tilde{\sigma},\tilde{\alpha}))$.

We will replace these equivalent pairs by smaller pairs, the first of
which we can relate to $(\sigma,\alpha)$ so as to apply Axiom (2). The
details are as follows.

For each $l \in J$ let $m_l$ be the minimum of $|\alpha^{-1}(l)|$ and
$|\tilde{\alpha}^{-1}(l)|$ and set $n_l:=|\sigma_1^{-1}(l)|$. Define
an injection $\pi$ from $[e] \sqcup \bigsqcup_{l \in J}([m_l] \sqcup
[n_l])$ to $[e] \sqcup S$ that is the identity on $[e]$, sends each
$[m_l]$ injectively to $m_l$ elements in $\alpha^{-1}(l) \subseteq S$
and sends each $[n_l]$ bijectively to the elements in
$\im(\tilde{\sigma})
\cap S$ where $\tau$ takes the value $l$. This construction
ensures that
$\cF(\pi)$ is defined at $(\tilde{\sigma},\tilde{\alpha})$. It might
{\em a priori} not be defined at $(\tilde{\sigma}',\tilde{\alpha}')$,
because $\im(\tilde{\sigma}') \cap S$ might not be contained in
$\im(\pi)$. But if it is not, then using Lemma~\ref{lm:Transposition}
we can replace $(\tilde{\sigma}',\tilde{\alpha}')$ by a
$\sim_S$-equivalent pair, with the same quadruple and with
the same bijection $g:J \to J'$, such that $\cF(\pi)$ {\em
is} defined at $(\tilde{\sigma}',\tilde{\alpha}')$. Now replace
$(\tilde{\sigma},\tilde{\alpha})$ and $(\tilde{\sigma}',\tilde{\alpha}')$
by their images under $\cF(\pi)$.

The point of distinguishing $\tilde{\cF}$ and $\cF''$ is that
these images may not be in $\tilde{\cF}$.  The reason is that while
$\tilde{v}+\ZZ_{\geq 0}^I$ is closed under the map that sends a pair
$w_1,w_2$ of vectors to the componentwise minimum vector
$\min(w_1,w_2)$ defined by $\min(w_1,w_2)(l):=\min(w_1(l),w_2(l))$, the
count vectors of pairs in $\tilde{\cF}$ with a fixed frequent set $J$ may
not quite be closed under componentwise minimum: the number of entries
that $\alpha$ has in the frequent set is not quite invariant under the
equivalence relations $\sim$, although it is up to a bounded difference.
This is remedied by allowing the componentwise minimum to have a count vector
in the bigger set $v+\ZZ_{\geq 0}^I$.

Hence the new $(\tilde{\sigma},\tilde{\alpha})$ and
$(\tilde{\sigma}',\tilde{\alpha}')$ are in $\cF''([e] \sqcup
\bigsqcup_{l \in J}([m_l] \sqcup [n_l]))$ and are related by the
same bijection $g$. Now construct similarly an injection $\pi':[e]
\sqcup \bigsqcup_{l \in J}([m_l] \sqcup [n_l]) \to [e] \sqcup T$
which is the identity on $[e]$, sends $[m_l]$ injectively into the set
$\alpha^{-1}(l)$, and sends $[n_l]$ bijectively to the set of elements
in $\im(\sigma) \cap T$ where $\tau$ takes the value $l$. Then
$(\tilde{\sigma},\tilde{\alpha})=\cF(\pi')((\sigma,\alpha))$.
Now apply Axiom (2) to find that there exists a pair
$(\sigma',\alpha') \sim_{[e] \sqcup T} (\sigma,\alpha)$ such that
$\cF(\pi')((\sigma',\alpha'))=(\tilde{\sigma}',\tilde{\alpha}')$.
The pair $(\sigma',\alpha')$ has the required quintuple. Moreover, since the pair
is $\sim_{[e] \sqcup T}$-equivalent to $(\sigma,\alpha)$, so
are their quintuples.
\end{proof}

\begin{prop}
The finite graph $G$ is a groupoid with objects its vertices (quadruples),
arrows as in Definition~\ref{de:GroupoidG}, and composition maps given
by composing the bijections $g$.
\end{prop}

Before proving this in general, we revisit
Example~\ref{ex:Model}.

\begin{ex}
Consider the functor $\cF:\FI \to \PF$ that maps $S$ to
\[ \cF(S)=\{(\sigma,\alpha) \mid \sigma \in \Hom_\FI(S_0,S),
\alpha \in [d]^{S \setminus \im(\sigma)}\} \]
where $S_0$ is a singleton and where we identify $[d]$ with $\ZZ/d\ZZ$ via
$a \mapsto a+d\ZZ$. Let $\sim_S$ be as in Example~\ref{ex:Model},
i.e. $(\sigma,\alpha) \sim (\sigma',\alpha')$ if and only
if either the pairs are equal, or else $\{j_0\}:=\im(\sigma)
\neq \im(\sigma')=:\{j_0'\}$ and for all $j \in S \setminus
\{j_0,j_0'\}$ we have $\alpha'(j)-\alpha'(j_0)=\alpha(j)$ and
$\alpha(j)-\alpha(j_0')=\alpha'(j)$ in $\ZZ/d\ZZ$.

In this case, $I=[d]$, and $v \in [d]^I$ is any vector containing
sufficiently many copies of each element of $[d]$.
Fix a pair $(\sigma,\alpha)  \in \cF''(T)$ where $\alpha$
has count vector in $v + \ZZ_{\geq 0}^I$. Set $\{j_0\}:=\im(\sigma)$. The
frequent set of $\alpha$ is $J=I=[d]$, and all equivalence classes in
$T$ under $\equiv$,
except possibly for that of $j_0$, are large, since the
fibres of $\alpha$ are large.

Pick a $j_0' \neq j_0$, let $\sigma'$ be the map $S_0 \to T$ with
image $\{j_0'\}$, and define $\alpha':T \setminus \{j_0'\} \to [d]$ via
$\alpha'(j)=\alpha(j)$ if $j \neq j_0,j_0'$ and
$\alpha'(j_0)=d$ (which is identified with $0+d\ZZ$). Then it
follows that $(\sigma,\alpha) \sim_T (\sigma',\alpha')$.  Since $\alpha'$
is defined at $j_0$ and takes the same value $d$ at least once more
(in fact, many times), we conclude that the equivalence class of $j_0$
under $\equiv$ is not a singleton, either. Hence the core of
$(\sigma,\alpha)$ is empty.

We can now complete the quintuple of $(\alpha,\sigma)$: the partially
defined map $\sigma_0$ from $S_0$ to the core $\emptyset$ is the
only such map, the surjection $\tau: T \to J$ equals $\alpha$
on $T \setminus \{j_0\}$ and $d$ on $j_0$, the map $\sigma_1$
maps the singleton $S_0$ to $d$, and $\overline{\alpha}$ is
the empty map $\emptyset \to \emptyset$. Note that the quadruple
$(J,\sigma_0,\sigma_1,\overline{\alpha})$ does not depend on the
particular choice of $(\sigma,\alpha)$, so the groupoid $G$
has a single object.

To determine all arrows in $G$, suppose that $(\sigma,\alpha) \sim_T
(\sigma'',\alpha'')$, set $\{j_0''\}:=\im(\sigma')$, and assume that
$j_0'' \neq j_0$ (otherwise we get the
identity arrow). Then for any $j \in T'':=T \setminus
\{j_0,j_0''\}$, we have
\[
\alpha''(j)-\alpha(j)=(\alpha''(j)-\alpha''(j_0))+(\alpha''(j_0)-\alpha(j))
=\alpha(j)+(\alpha''(j_0)-\alpha(j))=\alpha''(j_0)
\]
so that $\alpha''|_{T''}$ is obtained from $\alpha|_{T''}$ by a
coordinate-wise shift over the constant $\alpha''(j_0)=:a$. The map
$g_{(\sigma,\alpha),(\sigma'',\alpha'')}:[d] \to [d]$ is adding $a$
(modulo $d$). Conversely, for every choice of $\alpha''(j_0) \in
[d]=\ZZ/d\ZZ$, there is a unique pair $(\sigma'',\alpha'')$ equivalent to
$(\sigma,\alpha)$. We conclude that the groupoid $G$ is just (isomorphic
to) the group $\ZZ/d\ZZ$. 
\end{ex}

\begin{proof}
That $G$ has identity arrows follows from \eqref{eq:Groupoid2}, and
that all arrows are invertible follows from \eqref{eq:Groupoid1}
combined with \eqref{eq:Groupoid2}. It remains to check that
composition is well-defined. So let $g_1$ be an arrow in $G$ from
a vertex $q:=(J,\sigma_0,\sigma_1,\overline{\alpha})$ to a vertex
$q':=(J',\sigma_0',\sigma_1',\overline{\alpha'})$, and let $g_2$ be an
arrow in $G$ from $q'':=(J'',\sigma_0'',\sigma_1'',\overline{\alpha''})$
to $q$. We need to show that $g_1 \circ g_2$ is an arrow in $G$ from
$q''$ to $q'$. Now the existence of the arrow $g_2$ means that there
are a finite set $T$ and pairs $(\sigma,\alpha),(\sigma'',\alpha'')
\in \tilde{F}_e([e] \sqcup T)$ with core $[e]$ and quadruples $q,q''$,
respectively, such that $g_2=g((\sigma,\alpha),(\sigma'',\alpha''))$.

Let $(J,\tau,\sigma_0,\sigma_1,\overline{\alpha})$ be the quintuple of
$(\sigma,\alpha)$. By Proposition~\ref{prop:MatchSources}, $(J',g_1
\circ \tau,\sigma_0',\sigma_1',\overline{\alpha'})$ is the quintuple
of another pair $(\sigma',\alpha') \in \tilde{\cF}_e([e] \sqcup
T)$ with core $[e]$ and $(\sigma',\alpha') \simeq_{[e] \sqcup T}
(\sigma,\alpha)$. Then $g_1=g((\sigma',\alpha'),(\sigma,\alpha))$,
and $g_1 \circ g_2=g((\sigma',\alpha'),(\sigma'',\alpha''))$
is an arrow from $q''$ to $q'$, as desired.
\end{proof}

Consider the class of quintuples arising from pairs $(\sigma',\alpha')
\in \tilde{\cF}([e] \sqcup T), T \in \FI$ with core $[e]$. This class of quintuples
comes with a natural anchor map to the objects of $G$, namely, the map
that forgets $\tau$. The following says that equivalence classes of
quintuples are precisely orbits under $G$.

\begin{cor} \label{cor:GroupoidOrbits}
The groupoid $G$ acts on the set of quintuples of pairs in
$\tilde{\cF}([e] \sqcup T)$ with core $[e]$, for $T$ varying through
$\FI$, via the anchor map that sends a quintuple to its corresponding
quadruple. The orbits of this action are precisely the $\sim_{[e]
\sqcup T}$-equivalence classes. The action of $G$ commutes with the
action of $\Sym(T)$.
\end{cor}

\begin{proof}
The action of an arrow $g:(J,\sigma_0,\sigma_1,\overline{\alpha})
\to (J',\sigma'_0,\sigma'_1,\overline{\alpha'})$ on a quintuple
$(J,\tau,\sigma_0,\sigma_1,\overline{\alpha})$ yields the quintuple $(J',g
\circ \tau,\sigma'_0,\sigma'_1,\overline{\alpha'})$, and the axioms for
a groupoid action are readily verified. The action commutes with $\pi
\in \Sym(T)$ because $(g \circ \tau) \circ \pi=g \circ (\tau
\circ \pi)$. That the $G$-orbits on quintuples coming from
pairs $(\sigma,\alpha) \in \tilde{\cF}([e] \sqcup T)$ with
core $[e]$ are precisely the $\sim_{[e] \sqcup
T}$-equivalence classes follows from the last statement of
Proposition~\ref{prop:MatchSources}.
\end{proof}

Recall that, by Lemma~\ref{lm:SplitSym}, we need to count the elements
of $\tilde{\cF}([e] \sqcup T)$ with core $[e]$ up to $\sim_{[e] \sqcup
T}$ as well as up to the action of $\Sym([e]) \times \Sym(T)$.
Modding out the action of $\Sym(T)$ is now straightforward: it
just consists of replacing $\tau:T \to J$ by its count vector $u$
in $\ZZ_{\geq 0}^J$. Thus now the groupoid $G$ acts on quintuples
$(J,u,\sigma_0,\sigma_1,\overline{\alpha})$ coming from pairs in
$\tilde{\cF}_e([e] \sqcup T)$ with core $[e]$, where $u$ is a vector in
$\ZZ_{\geq 0}^J$. Also the group $\Sym([e])$ acts on such quintuples,
indeed, it acts on the corresponding quadruples and fixes $u$. We need to
count the quintuples up to the action of $G$ and $\Sym([e])$.  To do so,
we first observe that $\Sym([e])$ acts by automorphisms of $G$ on the
objects of $G$: if there is an arrow $g:q \to q'$, then there is also
an arrow $\pi(q) \to \pi(q')$ with the same label $g$, for each $\pi \in
\Sym([e])$. To stress that this arrow has a different source and target,
we write $\pi (g:q \to q') \pi^{-1}$ for that arrow $\pi(q)
\to \pi(q')$.

We now combine these actions into that of a larger groupoid $\tilde{G}$
with the same ground set as $G$ and with arrows $q \to q''$ all
pairs $(\pi,g:q \to q')$ where $g$ is an arrow from the quadruple $q$
to some quadruple $q'$ and $\pi \in \Sym([e])$ maps $q'$ to $q''$.
The composition $(\pi',g')\circ(\pi,g)$, where $g':q'' \to q'''$ and
$\pi' (q''')=q''''$, is defined as $(\pi' \circ \pi,\pi^{-1} (g':q''
\to q''')\pi \circ g)$. It is straightforward to see that $\tilde{G}$
is, indeed, a finite groupoid acting on quintuples. Now we are left to
count orbits of $\tilde{G}$ on quintuples.

\begin{prop}
There exists a quasipolynomial $f$ such that, for $n \gg 0$, the
number of orbits of $\tilde{G}$ on quintuples $(J,u \in \ZZ_{\geq 0}^J
,\sigma_0,\sigma_1,\overline{\alpha})$ arising from pairs $(\sigma,\alpha)
\in \tilde{\cF}_e([e] \sqcup [n])$ with core $[e]$ equals $f(n)$.
\end{prop}

\begin{proof}
Applying the orbit-counting lemma for groupoids (Lemma~\ref{lm:Burnside}),
it suffices to show that for each object of $\tilde{G}$, which is a
quadruple $q=(J,\sigma_0,\sigma_1,\overline{\alpha})$, and for each
arrow $(\pi,g):q \to q$ in $\tilde{G}$, the number of fixed points of
$(\pi,g)$ on quintuples with the given quadruple $q$ is quasipolynomial
for $n \gg 0$. Such a quintuple is fixed if and only if the count vector
$u \in \ZZ_{\geq 0}^J$ is fixed by $g$---indeed, $\pi$ acts trivially
on the count vector.

Now we figure out the structure of the set of count vectors arising
from such quintuples. First, if $u$ is the count vector of a quintuple
(with the fixed quadruple $q$) arising from $(\sigma,\alpha) \in
\tilde{\cF}_e([e] \sqcup [n])$, and if $l \in J$ and $i \in [n]$ such that
$\alpha(i)=l$, then $u-e_l$ is the count vector of the quintuple arising
from the pair $(\tilde{\sigma},\tilde{\alpha}):=\cF(\pi)((\sigma,\alpha))$,
where $\pi:[e] \sqcup [n-1] \to [e] \sqcup [n]$ is the identity on
$[e]$ and increasing from $[n-1] \to [n]$ and does not hit $i$. This
suggest that the set of count vectors that we are considering is downward
closed. However, it may be that $(\tilde{\sigma},\tilde{\alpha})$ is in
$\cF''([e] \sqcup [n-1]) \setminus \tilde{\cF}([e] \sqcup [n-1])$.

On the other hand, if we had started with $(\sigma,\alpha) \in \cF''([e]
\sqcup [n]) \setminus \tilde{\cF}([e] \sqcup [n])$ and applied such
an element $\cF(\pi)$ to it, the result would not have been
an element of $\tilde{\cF}([e] \sqcup [n-1])$.

This shows that the set of relevant count vectors is the difference
$N \setminus N'$, where $N$ and $N'$ are downward closed sets in
$\ZZ_{\geq 0}^J$. Hence, using the Stanley decomposition as in the
proof of Proposition~\ref{prop:BurnsideStanley}, the fixed
points $u$ that we are
counting are the lattice points in a finite disjoint union of rational
cones, each given as the intersection of a linear space (the eigenspace
of $g$ in $\RR^J$ with eigenvalue $1$) and a finite union of sets of the form
$u+\ZZ_{\geq 0}^{J'}$ with $J' \subseteq J$. The number of such $u$
is a quasipolynomial in $n=\sum_l u(l)$ for $n \gg 0$ for each of the
finitely many choices of $(\pi,g)$, hence so is their sum.
\end{proof}

This completes the proof of Theorem~\ref{thm:ModelFunctor}.

\end{proof}

\subsection{Pre-component functors} \label{ssec:PreComponent}

We will see that, if $X$ is a width-one $\FI$-scheme of finite type
over a Noetherian ring $K$, then we can cover $X(S)$ by means of closed,
irreducible subsets parameterised by a finite number of model functors
evaluated at $S$. Then, of course, the irreducible
components of $X(S)$
are among these. But to ensure that we are neither double-counting
components of $X(S)$ nor counting closed subsets that are strictly
contained in components, we need to keep track of the inclusions among
these closed subsets. At the combinatorial level, this is done using
compatible quasi-orders.

\begin{de}
Let $a \in \ZZ_{\geq 0}$, for each $b \in [a]$ let $k_b \in \ZZ_{\geq
0}$ and let $S \mapsto
\cF_b(S)/\!\!\sim_{b,S}$ be a model functor, where $\cF_b(S)$ is the set
of pairs $(\sigma,\alpha)$ with $\sigma:S_{0,b} \to S$ and $\alpha \in
\cE_b(S) \subseteq [k_b]^S$.

Suppose that we are given, for each $S \in \FI$, a quasi-order $\preceq_S$
(i.e., a reflexive and transitive relation) on the disjoint union
\[ \cF(S):=\bigsqcup_{b=1}^a \cF_b(S). \]
This collection of quasi-orders is called {\em compatible} if the
following properties are satisfied:
\begin{description}

\item[Compatibility (1)] whenever $\cF_b(S) \ni (\sigma,\alpha) \preceq_S (\sigma',\alpha')
\in \cF_{b'}(S)$, we have $b \leq b'$;

\item[Compatibility (2)] the restriction of $\preceq_S$ to each $\cF_b(S)$ equals the
equivalence relation $\sim_{b,S}$; and

\item[Compatibility (3)] for
all $\pi \in \Hom_\FI(S,T)$ and $(\sigma,\alpha) \in \cF_b(T)$ and
$(\sigma',\alpha') \in \cF_{b'}(T)$ with $\im(\pi) \supseteq \im(\sigma)\cup
\im(\sigma')$ we have
\[ (\sigma,\alpha) \preceq_T (\sigma',\alpha') \Rightarrow
\cF_b(\pi)((\sigma,\alpha)) \preceq_S
\cF_{b'}(\pi)((\sigma',\alpha')). \qedhere \]
\end{description}
\end{de}

Note that the condition that $\im(\pi)$ contains both $\im(\sigma)$
and $\im(\sigma')$ implies that $\cF_b(\pi)((\sigma,\alpha))$ and
$\cF_{b'}(\pi)((\sigma',\alpha'))$ are both defined.

\begin{de} \label{de:PreComponent}
In the setting of the previous definition, we introduce an
equivalence relation $\sim_S$ on $\cF(S)$ by $(\sigma,\alpha) \sim_S
(\sigma',\alpha')$ if both $(\sigma,\alpha) \preceq_S (\sigma',\alpha')$
and $(\sigma',\alpha') \preceq_S (\sigma,\alpha)$. Note that, by the
first and second axioms for pre-component functors, $(\sigma,\alpha) \in
\cF_b(S)$ is $\sim_S$-equivalent to $(\sigma',\alpha') \in \cF_{b'}(S)$
if and only if $b=b'$ and $(\sigma,\alpha) \sim_{b,S} (\sigma',\alpha')$.

The quasi-order $\preceq_S$ induces a partial order on the set
\[ \cF(S)/\!\!\sim_S=\bigsqcup_{b} (\cF_b(S)/\!\!\sim_{b,S})
\]
of equivalence
classes. We define the functor $\cC:\FI \to \PF$ on objects by
\[ \cC(S):=\{\text{the maximal elements of } \cF(S)/\!\!\sim_S\} \]
and on a morphism $\pi:S \to T$ as follows. Let $c \in \cC(T)$
be a maximal equivalence class. If $c$ contains some element
$(\sigma,\alpha)
\in \cF_b(T)$ at which $\cF_b(\pi)$ is defined (i.e., with $\im(\sigma)
\subseteq \im(\pi)$), and if, moreover,
$c':=[\cF_b(\pi)((\sigma,\alpha))]_{\sim_S}$ is also maximal in
$\cF(S)/\!\!\sim_S$, then we set
$\cC(\pi)(c):=c' \in \cC(S)$.
By the axioms for compatible quasi-orders, this
is independent of the choice of the representative
$(\sigma,\alpha)$ of $c$---subject to the requirement that
$\cF_i(\pi)$ be defined at that representative.

A functor $\cC:\FI \to \PF$ obtained in this manner is called a {\em pre-component
functor}.
\end{de}

\subsection{The final quasipolynomial count}
\label{ssec:FinalQuasi}

We retain the notation from \S\ref{ssec:PreComponent}: for each $b \in
[a]$ we have a model functor $S \mapsto \cF_b(S)/\!\!\sim_{b,S}$, and on
the disjoint unions $\cF(S):=\bigsqcup_{b} \cF_b(S)$ (for $S \in \FI$)
we have compatible quasi-orders $\preceq_S$.

\begin{thm} \label{thm:PreComponent}
Let $\cC$ be the pre-component functor corresponding to the data
above. Then the number of $\Sym(S)$-orbits on $\cC(S)$ is a
quasipolynomial in $|S|$ for all $S$ with $|S| \gg 0$.
\end{thm}

The proof requires the proof technique used in
\S\ref{ssec:SecondQuasi}, and takes up the rest of this
subsection.

\begin{proof}
By Theorem~\ref{thm:ModelFunctor}, we know that the number of
$\Sym(S)$-orbits on the disjoint union of model functors $\bigsqcup_{b}
(\cF_b(S)/\!\!\sim_{b,S})$ is eventually quasipolynomial in
$|S|$. From this disjoint union we will remove, for each $b \in [a]$, the
$\sim_{b,S}$-classes of pairs $(\sigma,\alpha)$ for which there exists
a $b'>b$ and a $(\sigma',\alpha') \in \cF_{b'}(S)$ with $(\sigma,\alpha)
\preceq_S (\sigma',\alpha')$.  We will see that, for a fixed $b \in [a]$,
the $\Sym(S)$-orbits on these deletions are also counted by a quasipolynomial.

To this end, fix $b \in [a]$ and let $M \subseteq \ZZ_{\geq 0}^{k_b}$
be the downward closed subset corresponding to $\cE_b$. From
\S\ref{ssec:SecondQuasi} we recall that, to prove that the number of
$\Sym(S)$-orbits on $\cF_b(S)/\!\!\sim_{b,S}$ is quasipolynomial in $|S|$
for $|S| \gg 0$, we performed induction on $M$ using Dickson's lemma. We do the
same here.

\begin{lm}
Fix $b \in [a]$.
For $|S| \gg 0$, the number of $\Sym(S)$-orbits on pairs $(\sigma,\alpha)
\in \cF_b(S)$ that are $\preceq$ some pair $(\sigma',\alpha') \in
\cF_{b'}(S)$ for some $b'>b$ is quasipolynomial in $|S|$.
\end{lm}

\begin{proof}
We construct the sub-functor $\cE_b' \subseteq \cE_b$
and the corresponding sub-functor $\cF_b' \subseteq \cF_b$
as in \S\ref{ssec:SecondQuasi}, as well as the difference
$\tilde{\cF}_b(S):=\cF_b(S) \setminus \cF'_b(S)$. By induction on $M$
using Dickson's lemma, we may assume that the lemma holds for the
sub-model functor $S \mapsto \cF_b'(S)/\!\!\sim_{b,S}$. On the other
hand, in \S\ref{ssec:SecondQuasi} we counted the $\Sym(S)$-orbits on
$\tilde{\cF}_b(S)/\!\!\sim_{b,S}$, for all $S$, as a finite sum of
orbit counts of certain finite groupoids. More precisely, for finitely
many values of a nonnegative integer $e$, we there considered the
pairs $(\alpha,\sigma) \in
\tilde{\cF}_b([e] \sqcup S)$ with core
equal to $[e]$. Each of these pairs gives rise to a quintuple $(J,u
\in \ZZ_{\geq 0}^J ,\sigma_0,\sigma_1,\overline{\alpha})$ where $J
\subseteq [k_b]$ is the frequent set of $\alpha$, and we showed that
the $\Sym(S)$-orbits on $\sim_{b,[e] \sqcup S}$-equivalence classes
of such pairs $(\alpha,\sigma)$ are in bijection with the orbits of a
certain groupoid on the sum-$|S|$ level set of a difference $N\setminus
N'$ where $N,N' \subseteq \ZZ_{\geq 0}^{k_b}$ are downward closed;
this difference is where the count vector $u \in \ZZ_{\geq 0}^J
\subseteq \ZZ_{\geq 0}^{k_b}$ lives.

Now suppose that $(\sigma,\alpha) \preceq_{[e] \sqcup S}
(\sigma',\alpha')$ for some $(\sigma',\alpha') \in \cF_{b'}([e]
\sqcup S)$ with $b' > b$, and let $i \in S$ be such that $l:=\alpha(i)
\in J$. If $i$ happens to be in $\im(\sigma) \cup \im(\sigma')$, then
choose $j \in S \setminus (\im(\sigma) \cup \im(\sigma'))$ such that
$\alpha(i)=\alpha(j)=l$---this can be done since $l$ appears frequently
in $\alpha$ (although, to be precise, when choosing the
vector $\tilde{v}$ in the construction of $\tilde{F}_b$, we now have to
make sure that its large entries values are large even compared to the
finitely many numbers $k_{b'}, b' \in [a]$).

By Lemma~\ref{lm:Transposition},
$(\sigma,\alpha) \sim_{b,[e] \sqcup S} (\tilde{\sigma},\tilde{\alpha}) := \cF_b((i\
j))((\sigma,\alpha))$, and the axioms for compatible quasi-orders imply
that $(\tilde{\sigma},\tilde{\alpha}) \preceq_{[e] \sqcup S}
(\tilde{\sigma}',\tilde{\alpha}'):=\cF_{b'}((i\
j))((\sigma',\alpha'))$. Moreover, we have achieved that $i \in S \setminus
(\im(\tilde{\sigma}) \cup \im(\tilde{\sigma}'))$.

Now let $\pi:[e] \sqcup S \setminus \{i\} \to [e] \sqcup S$ be the inclusion
map.  Then $\cF_b(\pi)$ and $\cF_{b'}(\pi)$ are defined at
$(\tilde{\sigma},\tilde{\alpha})$ and $(\tilde{\sigma}',\tilde{\alpha}')$,
respectively, and we have
\[ \cF_b(\pi)((\tilde{\sigma},\tilde{\alpha})) \preceq_{[e] \sqcup S
\setminus \{i\}}
\cF_{b'}(\pi)((\tilde{\sigma}',\tilde{\alpha}')). \]
The count vector in $\ZZ_{\geq 0}^J$ of the left-hand side
is just $u-e_l$.

We conclude that the set of count vectors of quintuples corresponding
to pairs in $\tilde{\cF}_{b,e}([e] \sqcup S)$ with core $[e]$
whose $\sim_{[e] \sqcup S}$-equivalence class is not maximal in
$\tilde{F}([e] \sqcup S)/\!\!\sim_{[e] \sqcup S}$
is downward-closed, or more precisely the intersection of a
downward-closed set with the earlier difference $N \setminus N'$ of
downward-closed sets, hence again a difference of downward-closed
subsets of $\ZZ_{\geq 0}^{k_b}$. The same orbit-counting argument with groupoids
as in \S\ref{ssec:SecondQuasi} applies, and shows that the number of
$\Sym(S)$-orbits on pairs in $\tilde{\cF}(S)$ that are not maximal in
$\preceq$ is a quasipolynomial.
\end{proof}

This concludes the proof of Theorem~\ref{thm:PreComponent}.

\end{proof}

\subsection{Component functors ``are'' pre-component functors}
\label{ssec:PreCompFIopScheme}

We are now ready to establish our main result on component
functors of width-one $\FIop$-schemes.

\begin{thm} \label{thm:FISchemePreComponent}
Let $X$ be a reduced and nice width-one affine $\FIop$-scheme of finite
type over a Noetherian ring $K$. Then there exists a pre-component
functor $\cC$ and a morphism $\phi: \cC \to \cC_X$ such that $\phi$
is an isomorphism at the level of species.

More precisely, there exist an $a \in \ZZ_{\geq 0}$, nonnegative
integers $k_b \in \ZZ_{\geq 0}$ for $b \in [a]$, model functors $S
\mapsto \cF_b(S)/\!\!\sim_{b,S}$, a collection of compatible quasi-orders
$\preceq_{S}$ on $\bigsqcup_b \cF_b(S)$ for each $S \in \FI$, and maps $\phi_b$
that assign to any $(\sigma,\alpha) \in \cF_b(S)$ an irreducible,
locally closed subset $\phi_b((\sigma,\alpha))$ of $X(S)$ such that,
for all $S$ and $T$,
\begin{enumerate}
\item $X(S)$ is the union of the sets $\phi_b((\sigma,\alpha))$ for $b
\in [a]$ and $(\sigma,\alpha) \in \cF_b(S)$; 

\item two locally closed sets $\phi_b((\sigma,\alpha))$ and
$\phi_{b'}((\sigma',\alpha'))$ with $(\sigma,\alpha) \in \cF_b(S)$
and $(\sigma',\alpha') \in \cF_{b'}(S)$ are either equal or disjoint;

\item $\cF_b(S) \ni (\sigma,\alpha) \preceq_S (\sigma',\alpha') \in
\cF_{b'}(S)$ holds
if and only if $\phi_b((\sigma,\alpha))$ is contained in the Zariski
closure of $\phi_{b'}((\sigma',\alpha'))$; and

\item if $(\sigma,\alpha) \in \cF_b(T)$ and $\pi \in
\Hom_\FI(S,T)$ are such that $\im(\pi) \supseteq \im(\sigma)$, so that
$(\sigma',\alpha'):=\cF_b(\pi)((\sigma,\alpha))$ is defined, then the
map $X(\pi)$ maps $\phi_b((\sigma,\alpha))$ dominantly into
$\phi_{b}((\sigma',\alpha'))$.
\end{enumerate}

The pre-component functor $\cC$ then assigns to $S$ the equivalence
classes of the maximal elements of $\preceq_S$, and the morphism $\phi$
is given by restricting the $\phi_b$ to these maximal elements and
taking the Zariski closure.
\end{thm}

The proof of this theorem takes up the rest of this
subsection.

\begin{proof}
Let $X=\Spec(B)$ be a width-one $\FIop$-scheme of finite type over a
Noetherian ring $K$. We assume that $X$ is reduced and nice. By the Shift
Theorem and Proposition~\ref{prop:Shift} there exist $S_0 \in \FI$ and
$h \in B(S_0)$ such that $X':=\Sh_{S_0} X [1/h]=\Spec(B')$ is of product
type in the sense of Definition~\ref{de:ProductType}. In particular,
$B'_0=B(S_0)[1/h]$ is a domain, $X'$ is isomorphic to $S \mapsto Z^S$
where $Z:=X'([1])$, and for each $S \in \FI$, each irreducible component
of $Z^S$ maps dominantly into $\Spec(B'_0)$.

For an $S \in \FI$, let $Z(S)$ be the open subset of $X(S)$ defined by
\[ Z(S):=\{p \in X(S) \mid \exists \sigma:S_0 \to S: (\sigma h)(p)
\neq 0\}. \]
Let $Y(S):=X(S) \setminus Z(S)$. Note that $Y \subsetneq X$ is a proper
closed $\FIop$-subscheme of $X(S)$, but $Z(S)$ is not (quite) functorial
in $S$: for $\pi \in \Hom_\FI(S,T)$ and $p \in Z(T)$ it might happen
that $X(\pi)(p)$ lies in $Y(S)$ rather than in $Z(S)$.

However, if $\sigma \in \Hom_\FI(S_0,T)$ and $\pi \in \Hom_{\FI}(S,T)$
satisfy $\im(\sigma) \subseteq \im(\pi)$, then $X(\pi)$ maps the
points of $Z(T)$ where $\sigma(h)$ is nonzero to points of $Z(S)$ where
$(\pi^{-1}\sigma)(h)$ is nonzero. Consequently, in spite of the fact that
$Z$ is not an $\FIop$-scheme, as in Definition~\ref{de:ComponentFunctor}
we associate to $Z$ the component functor $\cC_Z:\FI \to \PF$ that
assigns to $S$ the set of irreducible components of $Z(S)$ and to $\pi
\in \Hom_\FI(S,T)$ the partially defined map $\cC_Z(\pi):\cC_Z(T) \to
\cC_Z(S)$ that at a component $c \in \cC_Z(T)$ is defined and takes
the value $c' \in \cC_Z(S)$ if and only if $X(\pi)$ maps $c$ dominantly
into $c'$.

Now $\cC_X(S)$ is the union of $\cC_Z(S)$ and the components in $\cC_Y(S)$
that are not contained in the closure in $X(S)$ of any component of
$Z(S)$.

By Noetherian induction using Theorem~\ref{thm:Noetherianity}, we
may assume that Theorem \ref{thm:FISchemePreComponent} holds for
$Y \subsetneq X$, that is, there exists a morphism $\phi$ from a
pre-component functor $\cC_1$ to $\cC_Y$ that is an isomorphism at
the level of species, and $\cC_1$ is constructed from model functors $S
\mapsto \cF_b(S)/\!\!\sim_{b,S}, b \in [a]$ and a collection of compatible
quasi-orders $\preceq_S, S \in \FI$, while $\phi$ is constructed from
maps $\phi_b$ that assign (disjoint or equal) locally closed subsets in
$Y(S)$ to pairs $(\sigma,\alpha) \in \cF_b(S), b \in [a]$.

We now construct a model functor $\cC_2$ and a morphism $\cC_2 \to \cC_Z$
that is an isomorphism at the level of species.

Let $L$ be the fraction field of $B'_0$ and let $X'_{L}$ be the base
change of $X'$ to $L$. Since every irreducible component of $X'(S)$
maps dominantly into $\Spec(B'_0)$, the morphism $\cC_{X'_L} \to \cC_{X'}$
is an isomorphism. So for each $S \in \FI$, the irreducible components
of $X'_{L}(S)$ are in bijection with the irreducible components of $X(S_0
\sqcup S)[1/h] \subseteq Z(S_0 \sqcup S)$. Consequently, for each injection
$\sigma:S_0 \to S$ we now have an injective map
\begin{equation}  \label{eq:Xprime2Z}
\{\text{irreducible components of $X'_{L}(S \setminus \im(\sigma))$}\}
\to
\cC_Z(S),\quad c \mapsto X(\tilde{\sigma})(c)
\end{equation}
where $\tilde{\sigma}:S \to S_0 \sqcup (S \setminus \im(\sigma))$ is the
bijection that equals $\sigma^{-1}$ on $\im(\sigma)$ and the identity
on $S \setminus \im(\sigma)$. The image of this injective map
consists precisely of the irreducible components of $Z(S)$
on which $\sigma(h)$ is not identically zero.
As we vary $\sigma$, we thus obtain all irreducible components of $Z(S)$,
indeed typically multiple times.

Let $\overline{L}$ be a separable closure of $L$, and let
$X'_{\overline{L}}$ be the base change of $X'$ to $\overline{L}$.
Now we have a surjective morphism $\cC_{X'_{\overline{L}}} \to
\cC_{X'_{L}}$ whose fibres are Galois orbits. More precisely,
let $C_{1},\ldots,C_{k}$ be the irreducible components of
$X'_{\overline{L}}([1])=Z_{\overline{L}}$, and let $G$ be the image
of the Galois group $\Gal(\overline{L}/L)$ in $\Sym([k])$ through its
action on the components $C_{1},\ldots,C_{k}$. Then, as we have seen in
\S\ref{ssec:FiniteOverWide}, $\cC_{X'_{\overline{L}}}$ is isomorphic
to the functor $S \mapsto \cE(S):=[k]^S$.  Furthermore, $\cC_{X'}
\cong \cC_{X'_{L}}$ is isomorphic to the elementary model functor $S
\mapsto \cE(S)/G$.

Now we construct the functor $\cF:\FI \to \PF$ by
\[ \cF(S):=\{(\sigma,\alpha) \mid \sigma:S_0 \to S, \alpha \in
\cE(S \setminus \im(\sigma))\}. \]
Then, by the above, we have a surjective morphism
\[ \Psi:\cF \to \cC_Z; \]
concretely, $\Psi(S)$ takes $(\sigma,\alpha) \in \cF(S)$, computes the
component of $X'_{\overline{L}}(S \setminus \im(\sigma))$
corresponding to $\alpha$, its image in $\cC_{X'_{L}}(S \setminus
\im(\sigma))$ (modding out the Galois group), and then applies
the map \eqref{eq:Xprime2Z}. To simplify notation, we will write
$\Psi((\sigma,\alpha))$ instead of $\Psi(S)((\sigma,\alpha))$.

This surjection is by no means a bijection: even ignoring the Galois
groups for a moment, on the left we have pairs of a component in
$Z(S)$ and a specified $\sigma:S_0 \to S$ such that $\sigma(h)$ is
not identically zero on that component; and on the right we just have
components of $Z(S)$. A single component of $Z(S)$ may admit many
different such maps $\sigma$. We therefore introduce an equivalence
relation $\sim_S$ on $\cF(S)$ by $(\sigma,\alpha) \sim_S (\sigma',\alpha')
:\Leftrightarrow \Psi((\sigma,\alpha))=\Psi((\sigma',\alpha'))$.  In text
we will sometimes opress the $S$ and say that $(\sigma,\alpha)$ is {\em
equivalent} to $(\sigma',\alpha')$.

We will prove that the equivalence relation $\sim_{S}$
satisfies the axioms in the definition of a model functor
(\S\ref{ssec:ModelFunctors}).

\begin{lm} \label{lm:ax3}
Suppose that $\cF(T) \ni (\sigma,\alpha) \sim_T (\sigma',\alpha')
\in \cF(T)$ and let $\pi \in \Hom_\FI(S,T)$ with $\im(\pi)
\supseteq \im(\sigma) \cup \im(\sigma')$. Then
\[ \cF(\pi)((\sigma,\alpha)) \sim_S
\cF(\pi)((\sigma',\alpha')) \]
holds.
\end{lm}

\begin{proof}
The
assumptions assert that $\Psi((\sigma,\alpha))=\Psi((\sigma',\alpha'))$
and that $\cC_Z(\pi)$ is defined at this component of $Z(T)$.
By construction, we have
\[ \Psi(\cF(\pi)((\sigma,\alpha)))
= \cC_Z(\pi)(\Psi((\sigma,\alpha)))
= \cC_Z(\pi)(\Psi((\sigma',\alpha')))
= \Psi(\cF(\pi)((\sigma',\alpha'))), \]
so that $\cF(\pi)((\sigma,\alpha)) \sim_S
\cF(\pi)((\sigma',\alpha'))$, as desired.
\end{proof}

This lemma establishes Axiom (1) for the equivalence
relations $\sim_{S}$. We continue with Axiom (2).

\begin{lm}
Let $(\sigma,\alpha) \in \cF(T)$ and let $\pi \in \Hom_\FI(S,T)$ satisfy
$\im(\pi) \supseteq \im(\sigma)$. Assume that
$\cF(\pi)((\sigma,\alpha))
\sim_{S} (\sigma'',\alpha'') \in \cF(S)$. Set $\sigma':=\pi \circ
\sigma''$. Then there exists an $\alpha' \in \cE(T)$ such that
$(\sigma',\alpha')$ lies in $\cF(T)$, is $\sim_{T}$-equivalent
to $(\sigma,\alpha)$, and satisfies $\cF(\pi)((\sigma',\alpha'))=(\sigma'',\alpha'')$.
\end{lm}

\begin{proof}
Let $C:=\Psi((\sigma,\alpha))$ be the corresponding component of
$Z(T)$ and $D$ its image in $Z(S)$ under $X(\pi)$. The fact that
$\cF(\pi)((\sigma,\alpha)) \sim_{S} (\sigma'',\alpha'') \in \cF(S)$
implies that $\sigma''(h)$ is not identically zero on $D$. Then
$\sigma'(h)=\pi(\sigma''(h))$ is not identically zero on
$C$, and hence $C=\Psi((\sigma',\tilde{\alpha}'))$ for a suitable
$\tilde{\alpha}' \in \cE(T)$. Then
$\cF(\pi)((\sigma',\tilde{\alpha}'))=(\sigma'',\tilde{\alpha}'')$ for
$\tilde{\alpha}'':=\tilde{\alpha}' \circ \pi|_{S \setminus
\im(\sigma'')}$ and we have $(\sigma'',\tilde{\alpha}'') \
\sim_{S} (\sigma'',\alpha'')$. Since the first
component $\sigma''$ of these pairs is the same, it follows that
$\alpha''=g \tilde{\alpha}''$ for some $g \in G$. Now set
$\alpha':=g \tilde{\alpha}'$. As the action of $G$ commutes with
$\cF(\pi)$, we have
$(\sigma'',\alpha'')=\cF(\pi)((\sigma',\alpha'))$, and since
$\Psi((\sigma',\alpha'))=C$, we have $(\sigma',\alpha') \sim_{T}
(\sigma,\alpha)$.
\end{proof}

Next, we extablish Axiom (3) for model functors.

\begin{lm} \label{lm:Pattern}
Assume that $(\sigma,\alpha) \in \cF(T)$ is $\sim_{T}$-equivalent to
$(\sigma',\alpha') \in \cF(T)$. Then for all all $i,j \in T \setminus
(\im(\sigma) \cup \im(\sigma'))$ we have
\[ \alpha(i)=\alpha(j) \Leftrightarrow \alpha'(i)=\alpha'(j). \]
\end{lm}

\begin{proof}
Set $C:=\Psi((\sigma,\alpha))=\Psi((\sigma',\alpha'))$, an irreducible
component of $Z(T)$. Set $D:=X(\im(\sigma))[1/(\sigma h)]$, and
note that the inclusion $\im(\sigma) \to T$ yields a dominant
morphism $C \to D$. Similarly, the inclusion $\im(\sigma') \to T$
yields a dominant morphism $C \to D':=X(\im(\sigma'))[1/(\sigma'
h)]$.

Furthermore,  let $\tilde{D}$ the image of $C$ in $X(\im(\sigma) \cup
\im(\sigma'))[1/(\sigma h),1/(\sigma' h)]$ under the morphism coming
from the inclusion $\im(\sigma) \cup \im(\sigma') \to T$. Note that
$X(\sigma)$ maps $D$ isomorphically onto $\Spec(B_0')$ and $X(\sigma')$
maps $D'$ isomorphically onto $\Spec(B_0')$, and that $\tilde{D}$ maps
dominantly into $D$ and into $D'$.

Now let $M$ be the field of rational functions on $D$, and $M'$ the field
of rational functions on $D'$. Note that $\sigma$ and $\sigma'$ give rise
to isomorphisms from $L$ to $M$ and $M'$, respectively. We extend these
isomorphisms to isomorphisms from $\overline{L}$ to separable closures
$\overline{M}$ and $\overline{M'}$ of $M,M'$. This yields isomorphisms
from $\Gal(\overline{L}/L)$ to the Galois groups $\Gal(\overline{M}/M)$
and $\Gal(\overline{M'}/M)$.

The base change $C_{\overline{M}}$ equals
\[ \bigcup_{\beta \in G \alpha} \prod_{i
\in T \setminus \im(\sigma)} C_{\beta_i}; \]
here the product is over $\overline{M}$ and $C_{\beta_i}$ is regarded
as a variety over $\overline{M}$ via the isomorphism $\overline{L} \to
\overline{M}$; and similarly for $(\sigma',\alpha')$.  The base change
$\tilde{D}_{\overline{M}}$ splits as a similar union of products, but
now over $\beta \in G \alpha|_{\im(\sigma') \setminus \im(\sigma)}$.

Let $\tilde{M} \supseteq M \cup M'$ be the field of rational functions of
$\tilde{D}$, and let $\overline{\tilde{M}} \supseteq \overline{M} \cup
\overline{M'}$ be a separable closure of $\tilde{M}$.  The components
of the base change $C_{\overline{\tilde{M}}}$ can then be computed in
two different ways: by first doing a base change to $\overline{M}$ or
by first doing a base change to $\overline{M'}$. For the first route,
we have to analyse what happens to the product over $\overline{M}$
\[ \prod_{i \in T \setminus \im(\sigma)} C_{\beta_i} \] when doing
a base change to $\overline{\tilde{M}}$. Write this product as $V
\times_{\overline{M}} W$, where $V$ is the product over all $i \in
\im(\sigma') \setminus \im(\sigma)$, and hence an irreducible component
of $\tilde{D}_{\overline{M}}$; and $W$ is the product over all $i \in T \setminus
(\im(\sigma) \cup \im(\sigma'))$. The function field
$\tilde{M}$ of the irreducible $M$-scheme $\tilde{D}$ embeds into the
function field of any component of $\tilde{D}_{\overline{M}}$, hence
in particular into $\overline{M}(V)$, and this field
extension is algebraic, so that $\overline{\tilde{M}} \cong
\overline{\overline{M}(V)}$. So the base change to
$\overline{\tilde{M}}$ of
the product $V \times_{\overline{M}} W$ is just the base change of $W$ over
the separably closed field $\overline{M}$
with the field extension $\overline{\overline{M}(V)}$ and
hence still irreducible (e.g. by \cite[Tag 020J]{stacks-project}).

Summarising, we obtain a bijection from the irreducible components of
$C_{\overline{M}}$, which are labelled by elements of $G \alpha$, to
the irreducible components of $C_{\overline{\tilde{M}}}$,
and that bijection is evidently $\Sym(T \setminus
(\im(\sigma) \cup \im(\sigma')))$-equivariant. As a
consequence, for $i,j \in T \setminus (\im(\sigma) \cup \im(\sigma'))$ we have
$\alpha(i)=\alpha(j)$ if and only if the transposition $(i\ j)$
preserves some (and then each) component in $C_{\overline{M}}$, if and
only if that transposition preserves some (and then each) component in
$C_{\overline{\tilde{M}}}$.

Similarly, we obtain a bijection from the irreducible components of
$C_{\overline{M'}}$, which are labelled by elements of $G \alpha'$, to
the irreducible components of $C_{\overline{\tilde{M}}}$, with the same
remark about compatibility with the transposition $(i\ j)$. Combining
these results, we find that for $i,j \in T \setminus (\im(\sigma)
\cup \im(\sigma'))$ we have $\alpha(i)=\alpha(j)$ if and only if
$\alpha'(i)=\alpha'(j)$.
\end{proof}

We have now concluded the proof that $\cC_2:S \mapsto \cF(S)/\!\!\sim_S$ is
a model functor; and by construction, the map $\cC_2 \to \cC_Z$ induced
by $\Psi$ is a morphism that is an isomorphism at the level of species.

Finally, we combine the pre-component functor $\cC_1$ (mapping onto
$\cC_Y$) and the model functor $\cC_2$ (mapping onto $\cC_Z$) as
follows. Recall that $\cC_1$ is constructed by taking the equivalence
classes of maximal elements in $\bigsqcup_{b \in [a]} \cF_b(S)$.  We now
set $\cF_{a+1}:=\cF(S)$, and we extend $\preceq_S$ from $\bigsqcup_{b
\in [a]} \cF_b(S)$ to $\bigsqcup_{b \in [a+1]} \cF_b(S)$ by setting it
equal to $\sim_S$ on $\cF_{a+1}(S)=\cF(S)$, and for $(\sigma,\alpha)
\in \cF_b(S)$ with $b \leq a$ and $(\sigma',\alpha') \in \cF_{a+1}(S)$
we set $(\sigma,\alpha) \preceq_S (\sigma',\alpha')$ if and only if the
irreducible locally closed subset $\phi_b((\sigma,\alpha))$ of $Y(S)$
is contained in the Zariski closure of the irreducible component
$\phi_{a+1}((\sigma',\alpha')):=\Psi((\sigma',\alpha'))$ of $Z(S)$.
Properties (1)--(4) in Theorem~\ref{thm:FISchemePreComponent} are then
straightforward, and so are the properties Compatibility (1) and (2) from
the definition of pre-component functors. For instance, Compatibility
(1) follows from the fact that no component of $Z(S)$ can be contained
in $Y(S)$.

Regarding Compatibility (3): for $b=b'=a+1$ this is just Axiom (1), verified in Lemma~\ref{lm:ax3}; and for
$b \leq a, b'=a+1$ it follows from the simple fact that if the irreducible locally
closed set $\phi_b((\sigma,\alpha))$ in $Y(T)$ is contained in the Zariski
closure of the component $\Psi((\sigma',\alpha'))$ of $Z(T)$, then the
same holds for their projections in $Y(S)$ and $Z(S)$, respectively,
along the map $X(\pi)$.

With this quasi-order, the pre-component functor $\cC_1$ and the model
functor $\cC_2$ are combined into a pre-component functor $\cC_3$ with
an obvious morphism to $\cC_X$ that is an isomorphism at the level
of species.
\end{proof}

\subsection{Proof of the Main Theorem} \label{ssec:Proof}

\begin{proof}[Proof of Theorem~\ref{thm:Main}]
Before the Main Theorem we introduced $\FIop$-schemes slightly
differently than we do in \S\ref{sec:FISchemes}, but the two definitions
are equivalent via Remark~\ref{re:Sn}. So we may assume that $X$ is
a width-one $\FIop$-scheme of finite type over a Noetherian ring $K$
in the sense of \S\ref{sec:FISchemes}. Furthermore, since the Main
Theorem only makes a statement about the underlying topological space
of $X([n])$ for $n \gg 0$, we may assume that $X$ is both nice and
reduced.

By Theorem~\ref{thm:FISchemePreComponent}, the component functor $\cC_X$ of
$X$ is, at the level of species, isomorphic to a pre-component functor
$\cC$. In particular, for all $S$, the number of $\Sym(S)$-orbits on
$\cC_X(S)$ equals the number of $\Sym(S)$-orbits on $\cC(S)$.
By Theorem~\ref{thm:PreComponent}, this number is a
quasipolynomial in $|S|$ for all sufficiently large $S$. This
proves the Main Theorem and concludes our paper.
\end{proof}

\bibliographystyle{alpha}
\bibliography{draismajournal,diffeq,draismapreprint}

\end{document}